\documentclass[]{gAPA2e}

\begin{document}
\jvol{00} \jnum{00} \jyear{2011} \jmonth{September}



\title{A gap in the spectrum of the Neumann-Laplacian on a periodic waveguide}

\author{F.L. Bakharev$^{\rm 1}$,  S.A. Nazarov$^{\rm 2}$ and
K.M. Ruotsalainen$^{\rm 3}$ $^{\ast}$
\thanks{$^\ast$Corresponding author. Email: keijo.ruotsalainen@ee.oulu.fi}\\\vspace{6pt}  
$^{\rm 1}${\em{Chebyshev Laboratory, St.-Petersburg State University,\\
14th Line, 29b, Saint-Petersburg, 199178 Russia}};\\\vspace{6pt}
$^{\rm 2}${\em{Institute of Problems Mechanical Engineering,
Russian Academy of Sciences, St.Petersburg,
V.O., Bolshoy pr., 61, 199178 Russia}};\\\vspace{6pt}
$^{\rm 3}${\em{Mathematics Division, University of Oulu, P.O. Box 4500, 90014 Oulu, Finland}}\\\vspace{6pt}
\received{September 2011}}

\maketitle
\begin{abstract}
\paragraph*{Abstract} We will study the spectral problem related to the Laplace operator in a
singularly perturbed periodic waveguide. The waveguide is a quasi-cylinder with contains periodic
arrangement of inclusions. On the boundary of the waveguide we consider both Neumann and Dirichlet
conditions.
We will prove that provided the diameter of the inclusion is small enough in the spectrum of
Laplacian opens spectral gaps, i.e. frequencies that does not propagate through the waveguide.
The existence of the band gaps will verified using the asymptotic analysis of elliptic operators.
\end{abstract}

\begin{classcode}
35P05; 34E05; 34E10; 76M45\end{classcode}\bigskip

\section{Introduction}
\label{sec:1}

The main goal in this paper is to study the spectral properties of the
Neumann-Laplacian on the periodic singularly perturbed periodic quasi-cylinder
\(\Omega(\epsilon)\):
\begin{eqnarray}\label{form1}
-\Delta u^\epsilon(x)&=&\lambda^\epsilon u^\epsilon(x),\ x\in \Omega(\epsilon),\\
\label{form2}
\partial_n u^\epsilon(x)&=&0,\ x\in \partial\Omega(\epsilon).
\end{eqnarray}
This spectral problem should be interpreted as a singular perturbation of the
same problem in the straight cylinder \(\Omega=\omega\times{\mathbb{R}}\subset
{\mathbb{R}}^3\). 

The quasi-cylinder \(\Omega(\epsilon)\) is a periodic set which depends on a
small parameter \(\epsilon>0\). It is obtained from the straight cylinder
\(\Omega\) by a periodic nucleation of small voids with diameter of order
\(\epsilon\) (see fig. 1a and fig 1b, respectively).  
We remark that without lost of generality we have reduced the period to one by
rescaling.

\begin{figure}
\begin{center}
\resizebox{6cm}{4cm}{\includegraphics{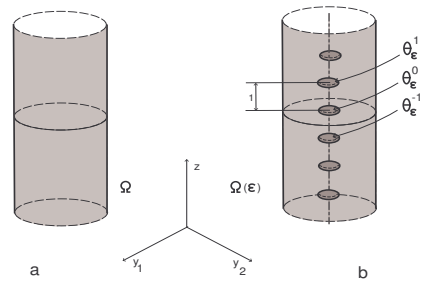}}
\caption{A cylinder and a quasi-cylinder with periodic family of voids}
\end{center}
\end{figure}

It is known (e.g. \cite{Wilcox}) that the spectrum \(\sigma\) of the
Neumann-Laplacian in the straight cylinder \(\Omega\) implies the continuous
spectrum \(\sigma_c\) which is the closed positive real semi-axis
\(\overline{\mathbb{R}_{+}}\), and the point spectrum is empty. In particular,
for any real \(\mu>0\) the functions, or \textit{oscillating waves}, 
\begin{equation}\label{wave1}
e^{\pm i\sqrt{\mu}z}, 
\end{equation}
give rise to a singular Weyl sequence of the problem operator \(A\) at
the point \(\mu\in \overline{\mathbb{R}_{+}}\) (cf. \S 9.1 in \cite{BiSo}).

The structure of the spectrum \(\sigma(\epsilon)\) in the perturbed quasi-cylinder
\(\Omega(\epsilon)\) is much more complicated. Owing to the Gel'fand transform
\cite{Gel} (see Section 2) 
and the Floquet-Bloch theory (see e.g. \cite{Kuchbook, NaPl}) 
the spectrum \(\sigma(\epsilon)\) of the Neumann-Laplacian is 
endowed with the band-gap structure. Namely, the essential spectrum
\(\sigma(\epsilon)=\sigma_e(\epsilon)\) is a union of closed connected segments
\(\Upsilon_n(\epsilon)\subset \overline{\mathbb{R}_{+}}\): 
\begin{equation*}
\displaystyle\sigma(\epsilon)=\bigcup_{n\in \mathbb{N}}\Upsilon_n(\epsilon).
\end{equation*}
We obtain \textit{a spectral gap} in the spectrum if there exist
an interval in the positive real semi-axis which is free of the spectrum, see
the left part of
fig. 2. However, when the segments overlap each other (see the right part of
fig.2), no spectral gap opens. 

Here we note that the essential spectrum \(\sigma_e(\epsilon)\) coincides with
the continuous spectrum provided none of the segments \(\Upsilon_n(\epsilon)\)
collapses to a single point, which becomes an eigenvalue of the operator
\(A(\epsilon)\) with infinite multiplicity, and belongs to the point spectrum
while the discrete spectrum is still empty. A negative answer to this
possibility is still an open question and the authors can
only prove the following: \textit{For any \(t>0\) there exists \(\epsilon_t>0\)
such that}
\(\sigma_p(\epsilon)\cap [0,t]=\emptyset\) when \(\epsilon\in (0,\epsilon_t]\). 
We do not pay any attention to this incomplete result and therefore always speak 
in this paper about the
essential spectrum.

\begin{figure}
\begin{center}\resizebox{8cm}{1cm}{\includegraphics{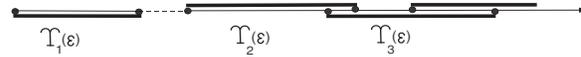}}\end{center}
\caption{The gap between the first and second bands,  while a gap between the second
and the fourth band is covered by the third one}
\end{figure}

In the literature the existence of spectral gaps is mainly investigated for
periodic media which are infinite in all directions while the coefficients of
differential operators, both in scalar and matrix case, are usually assumed
to be contrasting (see e.g. \cite{Fig,Fri,Gre,HemLin,Zhi} for scalar problems and \cite{Fil,na461} for systems).
Much less
results are obtained for periodic waveguides, which are infinite in one
direction only.  In this case spectral gaps ought to be opened by varying
the shape of the periodicity cell only which is in accordance with the results
in the known engineering practise.

There are two useful approaches to fulfil the goal. The first one utilises the
asymptotic analysis of the spectrum for the associated model problem in the
periodicity cell, like we do in the  present paper. This approach has been
realised for the Dirichlet-Laplacian in papers
\cite{Yoshitomi,FriSol,NaMaNo,CaNaPe} and others. We emphasise that our results
for the Neumann-Laplacian are completely new and require different techniques.

The second approach is based on the application of sharp parameter-dependent
estimates provided by Korn's inequalities, usually weighted and anisotropic,
see e.g. \cite{na435,na433,na454}  and others.
On one hand, this approach does not require a precise description of the dependence 
of the cell
shape perturbation on the small geometric parameter, but, in contrast to the first
approach, one can only prove the existence of one or several spectral
gaps without any information on the position and width of the gaps.

Combining of the above-mentioned approaches one could get much more
elaborated results by using their advantages. However, we cannot examine here
the whole spectrum and to our knowledge there does not exist any results of this
type. Here indeed we discuss only the first spectral gap.

To prove or to disprove the existence of a gap between the segments
\(\Upsilon_1(\epsilon)\) and \(\Upsilon_2(\epsilon)\) (cf. fig. 2), we employ
the method of compound asymptotic expansions
\cite{MaNaPl} to the associated model spectral problems in the periodicity cell 
\begin{equation}\label{cell}
\varpi_{\epsilon}=\varpi\setminus \overline{\theta_\epsilon^0},\quad
\varpi=\omega\times (0,1) 
\end{equation}
of the quasi-cylinder (\ref{quasi}). The straight cylinder 
with the small void \(\theta^0_\epsilon\) (see the two-dimensional dummy in fig.
3b) is nothing but a good example of a domain with the \textit{singularly
perturbed} boundary
\(\partial\varpi_{\epsilon}=\partial\varpi\cup\partial\theta_\epsilon^0\). The
asymptotic analysis of the eigenvalues for the Laplace operator have been
developed at a great extend in
\cite{Na61,Osava1,Gad,Na108,Na240,Na406,na416} and many others. 
However, the dependence of the model
problem on the dual variable \(\eta\in[0,2\pi)\) of the Gel'fand transform
brings a serious complication to both, the formal asymptotic procedures and the
justification since, instead of a single spectral problem, we have to deal with
a family of spectral problems dependent on the real parameter \(\eta\in
[0,2\pi)\) (see Section 3 and Section 4 for more details).

\begin{figure}
\begin{center}\resizebox{6cm}{2.5cm}{\includegraphics{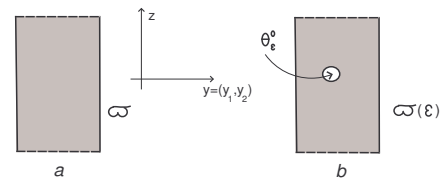}}\end{center}
\caption{The two-dimensional dummy of the periodicity cell}
\end{figure}

\section{Preliminaries and notations}
\label{seq:2}

\subsection{The quasi-cylinder and the Gel'fand transform}
\label{subseq:2-1}
The purpose of this section is to explain what do we mean by a
\textit{periodic quasi-cylinder}. There are several ways to describe it; but
here we start with the straight cylinder which is the cartesian product of a
bounded  Lipschitz domain \(\omega\subset\mathbb{R}^2\) and the real line:
\(\Omega=\omega\times\mathbb{R}\). 

Consider now  a fixed bounded smooth domain \(\theta\) with the coordinate
origin \(\xi=0\) in its interior. We introduce the family of sets
\(\theta^j_{\epsilon}\) by defining
\[\theta^j_{\epsilon}=\{x=(y_1,y_2,z):\ (y_1,y_2,z-j)\in \theta^0_\epsilon\}, 
\theta^0_{\epsilon}=\{x: \xi=\epsilon^{-1}(x-x^0)\in \theta\},\]
where \(x^0\in\varpi=\omega\times (0,1)\).
Note that the diameter of the set \(\theta^j_{\epsilon}\) is of order
\(\epsilon\).

The quasi-cylinder \(\Omega(\epsilon)\) is a periodic set which depends
on a small parameter \(\epsilon>0\) and it will be obtained
from the straight cylinder \(\Omega\) excluding the sets
\(\overline{\theta^j_\epsilon}\)
from it (see fig. 1a and fig 1b, respectively). 
In other words, the quasi-cylinder can be written as
\begin{equation}\label{quasi}
\Omega(\epsilon)=\Omega\setminus\bigcup_{j\in\mathbb{Z}}\overline{
\theta^j_\epsilon}.
\end{equation}

We call the set \(\varpi(\epsilon)=\varpi\setminus \theta^0_\epsilon\) 
in \eqref{cell} as the
periodicity cell of the quasi-cylinder. Note that the cylinder \(\Omega\) can be
viewed as a quasi-cylinder with the straight periodicity cell \(\varpi\) as in (\ref{cell}).

We recall here the definition of the Gel'fand transform \cite{Gel}: 
\begin{equation}
\label{Gelfand}
v(y,z)\mapsto\,V(y,z,\eta)=\frac{1}{\sqrt{2\pi}}\sum_{j\in \mathbb{Z}}
e^{-i\eta j}v(y,z+j),
\end{equation}
where \((y,z)\in \Omega(\epsilon)\) on the left, \(\eta\in[0,2\pi)\) and
\((y,z)\in\varpi(\epsilon)\) on the right. As it is well known this operator is
an isometric isomorphism from \(L^2(\Omega(\epsilon))\) onto
\(L^2(0,2\pi;L^2(\varpi(\epsilon)))\) and an isomorphism
between the Sobolev spaces
\(H^1(\Omega(\epsilon))\) and \(L^2(0,2\pi;H^1_\eta(\varpi(\epsilon)))\)
(see e.g. \cite{NaPl,Kuchbook}). 
Here \(L^2(0,2\pi;L^2(\varpi(\epsilon)))\) consists of
\(L^2(\varpi(\epsilon))\)-valued (complex) \(L^2\)-functions on \([0,2\pi]\).
Similarly, the space \(L^2(0,2\pi;H^1_\eta(\varpi(\epsilon)))\) contains those
\(H^1_\eta(\omega(\epsilon))\)-valued (complex) \(L^2\)-functions on
\([0,2\pi]\). Finally we have denoted by \(H^1_\eta(\varpi(\epsilon))\) the space
of Sobolev functions \(f\) which are quasi-periodic in the \(z\)-variable,
i.e. such that function \((y,z)\mapsto e^{-i\eta z}f(y,z)\)
is 1-periodic in \(z\). 

\subsection{The relationship between waves and Floquet waves in the straight
cylinder}
\label{subseq:2-2}

Before going into the asymptotic analysis of the spectral problem in the
periodic quasi-cylinder we will study the relationship of the standard wave
solutions and Floquet waves for the Neumann-Laplacian in the cylinder
\(\Omega\). They are the solutions of the differential equation
\begin{equation}\label{form3}
-\Delta_x u(x)=\lambda u(x),\ x\in\Omega\,,
\end{equation}
supplied with the Neumann conditions on the boundary
\begin{equation}\label{form4}
\partial_n u(x)=0,\ x\in\partial\Omega\,.
\end{equation}

Applying the Gel'fand transform to (\ref{form3}) and (\ref{form4}) we obtain
the model problem in the periodicity cell \(\varpi\) with the dual Gel'fand 
parameter  \(\eta\in[0,2\pi)\), namely, the differential equation 
\begin{equation}\label{form5}
-\Delta_x U(x,\eta)=\Lambda(\eta)U(x,\eta),\
(y,z)\in\varpi,
\end{equation}
the Neumann condition on the lateral boundary of \(\gamma=\partial\varpi\)
\begin{equation}\label{form6}
\partial_n U(x;\eta)=0,\ x\in \gamma,
\end{equation}
and the quasi-periodicity conditions on the opposite ends of the cylindrical cell
\(\varpi=\omega\times (0,1)\):
\begin{equation}\label{form7}
U(y,0)=e^{-i\eta}U(y,1),\ \partial_z U(y,0)=e^{-i\eta}\partial_z U(y,1),\
y\in\omega.
\end{equation}

Let \(\Lambda_p\) be an eigenvalue and let \(U_p\) be the
corresponding eigenfunction of the model spectral problem
\eqref{form4}-\eqref{form7},  
then the Floquet wave can be written as 
\begin{equation}
\label{Floq}
\Omega\ni (y,z)\mapsto e^{i\eta [z]}U_p(y,z-[z])
\end{equation}
where \([z]\in\mathbb{Z}\)
stands for the entire part of the real number \(z\). Notice that function
\eqref{Floq}
is smooth due to the quasi-periodicity conditions \eqref{form7} and it becomes a
solution of the homogeneous problem (\ref{form3}) with the spectral 
parameter \(\lambda=\Lambda_p\).

Any function of the type (\ref{Floq}) must be obtained from the
standard waves
\begin{equation}
\label{wave}
(y,z)\mapsto e^{i\zeta z}V_q(y)
\end{equation}
in the cylinder \(\Omega\), where \(\zeta\in \mathbb{R}\) is the dual Fourier
variable and \(V_q\) is an eigenfunction of the model problem 
\begin{equation}
\label{model-problem}
-\Delta_y V(y)=MV(y), \quad y\in\omega\,, 
\quad \partial_nV(y)=0, y\in \partial\omega
\end{equation}
on the cross-section \(\omega\). Notice that the
wave solution (\ref{wave}) with \(q=1\) and \(V_1(y)=const\) corresponds to the
eigenvalue \(M_1=0\) and coincides with the wave solution (\ref{wave1}) where
\begin{equation*}
\mu=\zeta^2
\end{equation*}
stands for the parameter in the Helmholtz operator \(\Delta_x+\mu\). The main
difference between formulae (\ref{Floq}) and (\ref{wave}) follows from the
role played by the parameters \(\eta\in[0,2\pi)\) and \(\zeta\in\mathbb{R}\).

We write the wave solution \eqref{wave} as follows:
\begin{equation}
\label{relation1}
e^{i\zeta z}V_q(y)
=e^{i(\zeta-2 \pi m )[z]}
e^{i(\zeta-2 \pi m)(z-[z])}e^{2\pi m iz}V_q(y)
\end{equation}
where \(m=[(2\pi)^{-1}\zeta]\) is the maximal integer such that \(2\pi m\leq
\zeta\).
Now we denote the first multiplier on the right of \eqref{relation1} by
\(e^{i\eta[z]}\) and the remaining part by \(U_p(y,z-[z])\) so that
\begin{eqnarray}
\label{rel2}
U_p(y,z)&=&e^{i\eta z}e^{2\pi miz}V_q(y),\\
\eta&=&\zeta-2\pi m\in [0,2\pi). \nonumber
\end{eqnarray}
Since \(e^{2\pi mi z}\) is 1-periodic in \(z\), we see that the function
\(U_p\)
satisfies the quasi-periodicity conditions \eqref{form7}. In other words,
\eqref{relation1}
implies the Floquet representation \eqref{Floq} of the wave solution
\eqref{wave}
with the ingredients \eqref{rel2}. 

Relationship (\ref{rel2}) helps us to list all eigenpairs
\(\{\Lambda_p,U_p\}\) of the model problem in the straight cell \(\varpi\) in
terms of the eigenvalues 
\begin{equation}
\label{Meig}
0=M_1<M_2\leq M_3\leq\ldots
\end{equation}
and eigenfunctions \(\{V_1,\,V_2,\,V_3,\dots\}\) of the model problem
\eqref{model-problem}.
Namely, there holds the relations
\begin{equation}
\label{eigrel} 
\Lambda_{q,m}(\eta)=M_q+(\eta+2\pi m)^2,\quad q\in{\mathbb{N}},\, m\in\mathbb{Z}.
\end{equation}
We emphasise that the eigenvalues \(\Lambda_{q,m}(\eta)\) are now enumerated
with two indexes so that it is necessary to reorder them to get the monotone increasing
unbounded sequence 
\begin{equation}
\label{seq-Lambda}
\Lambda_1(\eta)\leq \Lambda_2(\eta)\leq\cdots\leq \Lambda_k(\eta)\leq \cdots.
\end{equation}

The lowest dispersion curve (\ref{eigrel}) in fig. 4a,
corresponding to the smallest eigenvalue \(M_1=0\) of the Neumann problem in
\(\omega\), divides into an infinite family of smooth finite curves which by
above formulae are joined to the lattice in fig 4b.  In the sequel we are
mainly interested in the couple of lowest curves in the lattice which are
redrawn in fig 5.a, 
\begin{equation}
\label{dispcurve}
\Lambda^0_+(\eta):=\Lambda_{0,0}(\eta)=\eta^2,\
\Lambda^0_{-}(\eta):=\Lambda_{0,-1}(\eta)=(\eta-2\pi)^2 .
\end{equation}
We shall use the notation (\ref{dispcurve}) throughout the paper.

\begin{figure}
\begin{center}\resizebox{8cm}{3cm}{\includegraphics{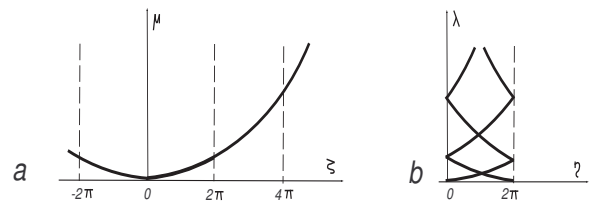}}\end{center}
\caption{The  lowest dispersion curve and the lattice frame made from it}
\end{figure}

The variational formulation of problem (\ref{form5})-(\ref{form7}) reads as 
\begin{equation}
\label{vari-form3}
(\nabla_xU,\nabla_x V)_{\varpi}=\Lambda(\eta) (U,V)_{\varpi}\,,
V\in H^1_{\eta}(\varpi)\,.
\end{equation}
Since for any real parameter \(\eta\) the bilinear form on the left-hand side
of (\ref{vari-form3}) is symmetric, closed and positive, we can associate with it 
a self-adjoint positive operator in the
Lebesque space \(L^2(\varpi)\) \cite[Thm. 10.1.2]{BiSo}. Moreover, the compact
embedding \(H^1(\varpi)\subset L^2(\varpi)\) and Theorems 10.1.5 and 10.2.2 in
\cite{BiSo} ensure that indeed the spectrum of this operator is discrete and
forms
an unbounded non-negative sequence (\ref{seq-Lambda}) which is constructed from
(\ref{Meig}). 

\subsection{Perturbation of the dispersion curves}
\label{subseq:14}

In the perturbed periodicity cell \(\varpi_{\epsilon}\) the both dispersion
curves (\ref{dispcurve}) experience a perturbation and, in principle, there can
occur two different situations pictured in fig. 5.b. and 5.c. First, the
dispersion curves \([0,2\pi) \ni \eta\to \Lambda_{\pm}^\epsilon(\eta)\) may
still intersect so that the spectral segments \(\Upsilon_1(\epsilon)\) and
\(\Upsilon_2(\epsilon)\) touch each other and the gap between them is closed.
The second possibility is that the graphs  of the functions
\begin{equation*}
\eta\mapsto\Lambda^\epsilon_{1}(\eta),\,
\eta\mapsto \Lambda^\epsilon_{2}(\eta)
\end{equation*}
do not intersect and the area overshadowed in fig. 5c projects into a gap
between segments \(\Upsilon_1(\epsilon)\) and \(\Upsilon_2(\epsilon)\). In the
sequel we will describe when the gap appears and indicate asymptotically its
position and width.

\begin{figure}
\begin{center}\resizebox{8cm}{3cm}{\includegraphics{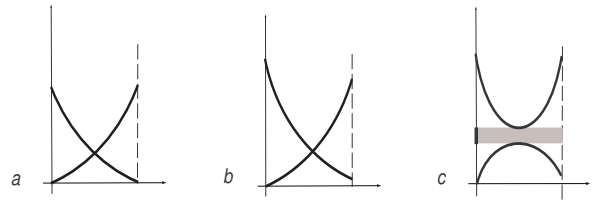}}\end{center}
\caption{Unperturbed (a), intersecting (b) and disjoint (c) curves}
\end{figure}

Before proceeding with the problem statement, the formal asymptotic analysis and
the justification, we note that the gap opens only under the condition 
\begin{equation}
\label{cond1}
M_1+\pi^2<M_2.
\end{equation}
Otherwise, the segment \(\Upsilon_2(\epsilon)\), which corresponds to the
perturbed eigenvalues in (\ref{eigrel}) with \(q=2\) and \(m=0,-1\) covers the
gap opened between the segments \(\Upsilon_1(\epsilon)\) and
\(\Upsilon_3(\epsilon)\) (see fig.6) which are now generated by eigenvalues in
(\ref{eigrel}) with \(q=1\) and \(m=0,-1\).

\begin{figure}
\begin{center}\resizebox{7cm}{3cm}{\includegraphics{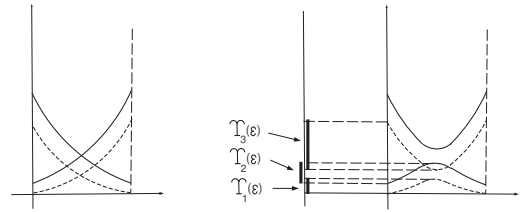}}\end{center}
\caption{Overlapping bands. The dotted lines correspond to the perturbation of
the lowest curves in fig. 5.a.}
\end{figure}

However, the rescaling performed in the very beginning rescales
also eigenvalues. To see this, assume that the period \(l\)
of the void nucleation is bigger than
one. Now the rescaling diminishes the size of the cross-section, i.e. the change
of the variables 
\((y,z)\to(y',z')=(l^{-1}y,l^{-1}z)\)
yields the shrunken cross-section
\[\omega'=\{y'\in \mathbb{R}^2:\ ly'\in\omega\}.\]
Simultaneously, eigenvalues of the model problem in \(\omega\) changes into
the eigenvalues \(M_1'=0\) and \(M_2'=l^2M_2\) of the model problem in the
shrunken cross-section \(\omega'\). By choosing \(l\) large enough we fulfil 
the condition (\ref{cond1}) for \(M_1'\). In other words, posing the small voids
\(\theta_{\epsilon}^{0}\) with a sufficiently large period may open 
a gap in the essential spectrum, where as if the period of the voids is
sufficiently small first triple of the spectral bands surely overlap.

Assuming condition (\ref{cond1}) we study the spectral problem 
\eqref{form1}-\eqref{form2}. Applying the Gel'fand transform 
(\ref{Gelfand}) to problem (\ref{form1})-(\ref{form2}), 
we obtain the model spectral problem in the periodicity cell
\(\varpi_{\epsilon}\) depending on \(\eta\in [0,2\pi)\):
\begin{eqnarray}\label{spctrlprob1}
-\Delta U^\epsilon(y,z)&=&\Lambda^\epsilon(\eta)U^\epsilon(y,z),\
(y,z)\in\varpi_{\epsilon}\,,\\
\label{bndrycond}
\partial_n U^\epsilon(x)&=&0,\ x\in
\gamma\cup\partial\theta(\epsilon)\,,\\
\label{percond}
U^\epsilon(y,0)=e^{-i\eta}U^\epsilon(y,1),&&\partial_z
U^\epsilon(y,0)=e^{-i\eta}\partial_z
U^\epsilon(y,1),\
y\in\gamma_0.
\end{eqnarray}

As above, using Theorem 10.1.2 \cite{BiSo} we associate with these spectral
problems a family of unbounded positive self-adjoint operators
\(T^\epsilon(\eta)\) on \(L^2(\varpi_{\epsilon})\) with discrete spectra, i.e.,
the eigenvalue sequences 
\[0\leq \Lambda^\epsilon_{1}(\eta)\leq\Lambda^\epsilon_{2}(\eta)\leq \cdots\leq
\Lambda^\epsilon_{k}(\eta)\leq \cdots\rightarrow\infty.\]
The immediate objective becomes to construct asymptotics for these
eigenvalues as \(\epsilon\to+0\).

\subsection{Mixed boundary value problem}

Although all the calculations are presented for the 
Neumann conditions on the boundary of the quasi-cylinder \(\Omega(\epsilon)\), we
also discuss slight modifications, which are needed to adapt our
calculations to the mixed boundary value problem 
\begin{eqnarray}\label{dir-1}
-\Delta u^\epsilon(x)&=&\lambda^\epsilon u^\epsilon(x),\
x\in \Omega(\epsilon)\,,\nonumber\\
\partial_n u^\epsilon(x)&=&0,\ x\in
\partial\Omega(\epsilon)\setminus \partial \Omega=\cup_{j\in\mathbb{Z}}\partial
\theta^j_\epsilon\,,\\ 
u^\epsilon(x)&=&0, x\in \partial \Omega\,.\nonumber
\end{eqnarray}                                      
We emphasise that the
asymptotic procedure of \cite{MaNaPl} employed here is not sensitive to the type
of boundary conditions on the lateral part of the boundary
\(\partial\omega\times (0,1)\) of the periodicity cell is employed. 
However, the change of the boundary condition from
Neumann
to Dirichlet on the surfaces of small voids crucially simplifies the
asymptotic analysis\footnote[1]{This does not hold in the two-dimensional
case, where the existence of a gap is still fully an open question.}. In this
way our paper becomes much more technical than \cite{NaMaNo} and \cite{CaNaPe},
where the Dirichlet problem was considered in the regularly and singularly
perturbed periodicity cells. The particular conclusions on the geometrical
characteristics of the gap are also different.

\section{The formal asymptotic procedure}
\label{sec: 3}

In this section we present a calculation which furnishes an asymptotic
formula for the first two eigenvalues (\ref{dispcurve}).
In order to simplify the presentation we start with considering the case when the eigenvalues for
the model problem are simple, e.g. when the Gel'fand dual variable \(\eta\) is
not equal to \(\pi\). Notice that the dispersion curves in fig. 5a
intersect each other just at \(\eta=\pi\), and we examine this situation
in the end of this section. 

In Sections 3.3 and 3.4 we give detailed calculations of the
boundary layers and a regular terms which lead to a direct formula for the main 
correction term in the asymptotics of the eigenvalues. 
With this asymptotic formula we are then able
to consider also the case \(\eta\approx \pi\) although at \(\eta=\pi\) where the model
problem has a multiple eigenvalue. As a result, in Section 3.5 we derive a
system of two linear algebraic equations from which we can find out the
perturbation terms in the asymptotics of the eigenvalues when
\(\eta\) approximately equal to \(\pi\).

To make the necessary conclusion for the case \(\eta\approx \pi\), we will apply
the trick proposed in \cite{NaMaNo} and introduce the additional deviation
parameter \(\psi\) in the representation
\begin{equation}
\label{psi}
\eta=\pi+\psi\epsilon^3\,.
\end{equation}
With the help of this new variable \(\psi\in[-\psi_0,\psi_0]\) we can
write the asymptotic formulae at the \(c\epsilon^3\)-neighbourhood of the point
\(\eta=\pi\).

We emphasise here the following difference in the notation: In Sections
(\ref{subsec: 31})-(\ref{subsec:34}) we set
\begin{eqnarray}
 \label{eta1}
\Lambda_+^{\epsilon}(\eta)&=&\Lambda_1^{\epsilon}(\eta),\,\Lambda_{-}^{\epsilon}(\eta)=\Lambda_2^{\epsilon}(\eta),\,
\eta\in [0,\pi),\\
\Lambda_+^{\epsilon}(\eta)&=&\Lambda_2^{\epsilon}(\eta),\,\Lambda_{-}^{\epsilon}(\eta)=\Lambda_1^{\epsilon}(\eta),\,
\eta\in (\pi,2\pi),\nonumber
\end{eqnarray}
while in Section (\ref{subsec:35})
\begin{equation}\label{eta2}
\Lambda_+^{\epsilon}(\eta)=\Lambda_2^{\epsilon}(\eta),\,\Lambda_{-}^{\epsilon}(\eta)=\Lambda_1^{\epsilon}(\eta),\,
\eta=\pi.
\end{equation}
The evident difference in the notations (\ref{eta1}) and (\ref{eta2}) helps to keep similar
asymptotic formulae in both cases.

Calculations in Sections \ref{subsec: 31}-\ref{subsec:34} are given for the Neumann problem 
(\ref{form1})-(\ref{form2}) in such a way that they can easily 
adopted for the mixed boundary value problem (\ref{dir-1}).
Namely, the eigenpair of the problem (\ref{model-problem}) \(\{M_1=0\), \(V_1=const\}\)  must be replaced
by the first eigenpair \(\{M_1, V_1\}\) of the Dirichlet problem
\begin{equation}
\label{S4}
-\Delta_y V(y)=M V(y),\, y\in \omega,\qquad V(y)=0, \, y\in\partial \omega\,.
\end{equation} 
Notice that the eigenvalue \(M_1\) is positive and simple while the corresponding eigenfunction \(V_1\) 
can be chosen positive inside the cross-section \(\omega\).

\subsection{Asymptotic expansions}
\label{subsec: 31}

To describe the asymptotic behaviour of two first eigenvalues
\(\Lambda^\epsilon_{\pm}(\eta)\), 
we accept the following \textit{ansatz}:
\begin{equation}\label{ansatz1}
\Lambda^\epsilon_\pm(\eta)=\Lambda^0_\pm(\eta)+\epsilon^3\Lambda'_\pm(\eta)+
\widetilde{\Lambda}^\epsilon_\pm(\eta),
\end{equation}
where \(\Lambda^0_\pm(\eta)\) is the eigenvalue (\ref{dispcurve}) of problem
(\ref{form3})-(\ref{form5}),
\(\Lambda'_\pm(\eta)\) is a correction term and \(\widetilde{\Lambda}
^\epsilon_\pm(\eta)\) a small remainder to be evaluated and estimated.

According to general asymptotic method in domains with singular
perturbations of the boundaries (see e.g. \cite{MaNaPl}) the
corresponding asymptotic ansatz for eigenfunctions reads as follows 
\begin{eqnarray}
\label{ansatz2}
U^{\epsilon}_\pm(x;\eta)&=&U^0_\pm(x;\eta)+\epsilon
\chi(x)W^1_\pm(\epsilon^{-1}(x-x^0);\eta)\nonumber\\
&+& \epsilon^2\chi(x)W^2_\pm(\epsilon^{-1}(x-x^0);\eta)
+ \epsilon^3U'_\pm(x;\eta)+\widetilde{U}^\epsilon_\pm(x;\eta).
\end{eqnarray}
The functions \(W^j_\pm\) are of boundary layer type and \(\chi\) is a smooth
cut-off function which equals one in a neighbourhood of the point \(x^0\) and
vanishes in the vicinity of the boundary \(\partial\varpi\). The functions
\(U^0_\pm\) are the first two eigenfunctions of the unperturbed problem (cf. formula (\ref{rel2})), namely
\begin{equation}
\label{eigen2}
U^0_{+}(x;\eta)=e^{i\eta z}V_1(y), \quad 
U^0_{-}(x;\eta)=e^{i(\eta-2\pi) z}V_1(y), 
\end{equation}
where \(V_1\) is the first
eigenfunction of the Laplace operator on the cross-section, normalised
in \(L^2(\varpi)\). Clearly, both the functions \eqref{eigen2}
meets the quasi-periodicity conditions \eqref{form7}. 
Note that in the case of the Neumann boundary conditions on
\(\partial \Omega\) the function \(V_1\) is just a constant. One readily
verifies that since \(M_1=0\) is the first eigenvalue of the Laplacian on the
cross-section in the Neumann case, the first couple of eigenvalues of the
unperturbed problem take the form (\ref{dispcurve}).

\subsection{The boundary layers}
\label{subsec: 32}

The boundary layer \(W^1_\pm\) depends on the  ``fast''
variables (``stretched`` coordinates)
\[\xi=(\xi_1,\xi_2,\xi_3)=\epsilon^{-1}(x-x^0)\] 
and is needed to compensate for that the main regular term \(U^0_\pm\)
in the expansion does not fulfil the Neumann type boundary condition
\[\partial_{n(x)}u(x)=0,\, x\in \partial\theta_\epsilon.\] 
By the Taylor formula applied to the function \(U^0_{\pm}\) in the ''slow``
variables \(x\), we have 
\[\partial_{n(\xi)} U^0_{\pm}(x;\eta)=
n(\xi)\cdot \nabla_xU^0_{\pm}(x^0)+O(\epsilon),\,x\in\partial \theta_\epsilon.\] 
Thus \(W^1_\pm\) ought to be a solution of the exterior Neumann problem
\begin{eqnarray}\label{form10}
\Delta_\xi W^1_\pm(\xi)&=&0,\ \xi\in\mathbb{R}^3\setminus \bar{\theta}\,,\\
\partial_{n(\xi)} W^1_\pm(\xi)&:=&G^1_{\pm}(\xi)=-n(\xi)\cdot
\nabla_xU^0_{\pm}(x^0;\eta),\
\xi\in\partial\theta.\nonumber
\end{eqnarray}
Since \(\int_{\partial \theta}n_k(\xi)\,ds(\xi)=0\) for \(k=1,2,3,\)
the function \(G^1_{\pm}\) satisfies the orthogonality condition
\[\int_{\partial\theta}G^1_{\pm}(\xi)\,ds_\xi=0.\]
Thus, the problem (\ref{form10}) admits a solution which decays faster at
infinity than the fundamental solution \(\Phi(\xi)=(4\pi|\xi|)^{-1}\)
of the Laplace operator in \({\mathbb R}^3\). Moreover, the solution has the
asymptotic behaviour 
\begin{equation}\label{form11}
W^1_\pm(\xi)= \sum_{j,k=1}^{3}Q_{jk}\frac{\partial\Phi}{\partial
\xi_k}(\xi) \frac{\partial U^0_{\pm}}{\partial x_j}(x^0;\eta)+O(|\xi|^{-3})
\end{equation}
where \(Q=(Q_{jk})_{j,k=1,2,3}\) is the matrix of the 
virtual mass for the set \(\overline{\theta}\subset\mathbb{R}^3\), 
which is symmetric and negatively definite for any subset \(\theta\) with a positive volume
(see \cite[Appendix G]{PolyaSzego}).

\subparagraph{Remark 3.1} As shown in \cite[Appendix G]{PolyaSzego} (see also \cite{MaNaPl}) the representation
\begin{equation}
 \label{QQ}
Q=-\mathbb{I}_3\textrm{meas}_3(\theta)+\widetilde Q
\end{equation}
is valid, where \(\mathbb{I}_3\) is the \(3\times 3\) identity matrix and \(\widetilde Q\) is non-positive.
The latter is still negative definite in the case \(\textrm{meas}_3(\theta)\neq 0\); 
but degenerates for a straight crack.

In order to find the appropriate boundary value problem for
the second boundary layer term \(W^2_\pm\) we take into account a quadratic
polynomial in the Taylor formula for \(U_\pm^0\):
\[
\partial_a U^0_\pm(x;\eta)=\sum_{j=1}^3 a_j\frac{\partial U^0_\pm}{\partial
x_j}(x^0;\eta)+
\frac{1}{2}\sum_{j,k=1}^3\partial_a(x_jx_k)
\frac{\partial^2U^0_\pm}{\partial x_j\partial x_k}(x^0;\eta)+O(|x-x^0|^2).
\]

In this way, we conclude that the 
second boundary layer term \(W^2_\pm\) should satisfy the following
exterior Neumann problem:
\begin{eqnarray*}\label{form12}
-\Delta_\xi W^2_\pm(\xi)&=& 0,\
\xi\in \Xi=
\mathbb{R}^3\setminus \overline{\theta},\\
\partial_{n(\xi)} W^2_\pm(\xi)&=&-\frac12\sum_{j,k=1}^{3}
\frac{\partial(\xi_j\xi_k)}{\partial n(\xi)}
\frac{\partial^2U^0_\pm}{\partial x_j\partial x_k}(x^0;\eta),\
\xi\in\partial\theta.\nonumber
\end{eqnarray*}
The function \(W^2_\pm\) admits the representation 
\begin{equation}
\label{number}
W^2_\pm(\xi)= c_{\pm}\rho^{-1}+O(\rho^{-2}), \quad \rho=|\xi|\to \infty\,.
\end{equation}
To calculate the coefficient \(c_\pm\) in (\ref{number}) we use the
Green formula in the domain \(\Xi_R=B_R(0)\setminus \overline{\theta}\) 
where \(B_R(0)\) is a ball of radius \(R\) centred at \(\xi=0\) to
obtain
\begin{eqnarray}
\label{calc-c-1}
0=\int_{\Xi_R}\Delta_\xi W^2_\pm(\xi)\,d\xi
&=& \int_{\partial \Xi_R}\partial_{n(\xi)}W^2_{\pm}(\xi)ds_\xi\nonumber\\
&=&\int_{\mathbb{S}_R}\frac{\partial W^2_{\pm}}{\partial \rho}(\xi)ds_\xi
+\int_{\partial\theta}\partial_{n(\xi)}W^2_{\pm}
(\xi)ds_\xi\nonumber\\
&=&\int_{\mathbb{S}_R}\frac{\partial W^2_\pm}{\partial\rho}(\xi)\, ds_\xi\\
&-&\frac12\sum_{j,k=1}^3\frac{\partial^2 U^0_\pm}{\partial x_j\partial x_k}(x^0)
\int_{\partial\theta}
\frac{\partial(\xi_j\xi_k)}{\partial n(\xi)}\,ds_\xi \nonumber\,.
\end{eqnarray}
Notice that on the sphere \(\mathbb{S}_R=\partial B_R(0)\) we have
\(\partial_{n(\xi)}=\partial_\rho\).

Applying the asymptotic expansion (\ref{number}), we see that in
(\ref{calc-c-1})
\begin{equation*}
\label{4pic}
\lim_{R\to+\infty}\int_{\mathbb{S}_R}\frac{\partial
W^2_\pm}{\partial\rho}(\xi)\, ds_\xi
= \lim_{R\to+\infty}\int_{\mathbb{S}_R}
\left(-\frac{c_\pm}{\rho^2}+O(\rho^{-3})\right)\,ds_\xi
= -4\pi c_{\pm}\,.
\end{equation*}
Using the Green formula again but now inside \(\theta\), note that the normal
becomes inward then, we derive the relation
\begin{eqnarray*}
\label{I_3}
&-&\frac{1}{2}\sum_{j,k=1}^3\frac{\partial^2 U^0_\pm}{\partial x_j\partial
x_k}(x^0)\int_{\partial\theta}
\partial_{n(\xi)}(\xi_j\xi_k)\,ds_\xi\nonumber\\
&=&\frac{1}{2}\sum_{j,k=1}^3\frac{\partial^2 U^0_\pm}{\partial x_j\partial
x_k}(x^0)\int_{\theta}
\Delta(\xi_j\xi_k)\,d\xi
=\sum_{j=1}^3\frac{\partial^2 U^0_\pm}{\partial x_j^2}(x^0)\textrm{meas}_3(\theta)\\
&=&\Delta U^0_\pm(x^0)\textrm{meas}_3(\theta)=-
\Lambda_\pm^0(\eta) U^0_\pm(x^0)\textrm{meas}_3(\theta)\,.\nonumber
\end{eqnarray*}
In the last equality we also have used the fact that \(U^0_\pm\) is
the eigenfunction of the unperturbed problem corresponding to the eigenvalue
\(\Lambda^0_{\pm}(\eta)\) ( see (\ref{form5})).

Collecting the above results we finally
obtain the formula 
\begin{equation}
\label{calc-c-4}
4\pi c_\pm=-\Lambda^0_{\pm}(\eta) U^0_\pm(x^0)\textrm{meas}_3\theta.
\end{equation}

\subsection{The regular correction}
\label{subsec:33}

For the calculation of the main correction term \(\Lambda_{\pm}'(\eta)\)
in the asymptotic expansion of \(\Lambda^\epsilon_\pm(\eta)\) we substitute the
asymptotic ans\"atze (\ref{ansatz1}) and (\ref{ansatz2}) into the spectral
problem (\ref{spctrlprob1})-(\ref{bndrycond}). By (\ref{form11}), (\ref{number}) and (\ref{calc-c-4})
the boundary layer terms \(W^1_\pm(\epsilon^{-1}(x-x^0);\eta)\) and
\(W^2_\pm(\epsilon^{-1}(x-x^0);\eta)\) have the following behaviour:
\begin{eqnarray}
\label{W1x}
W^1_\pm\left(\frac{1}{\epsilon}(x-x^0);\eta\right)&=&
\epsilon^2\sum_{j,k=1}^{3} Q_{jk} \frac{\partial \Phi}{\partial x_k}(x-x^0)\frac{\partial U^0_\pm}{\partial
x_j}(x^0;\eta)\\
&+&O(\epsilon^3|x-x^0|^{-3}),\nonumber
\end{eqnarray}
\begin{equation}
\label{W2x}
W^2_\pm(\epsilon^{-1}(x-x^0);\eta)=\epsilon c_\pm |x-x^0|^{-1}
+O(\epsilon^2|x-x^0|^{-2})\,.
\end{equation}
Hence we notice that in (\ref{ansatz2}) the contribution of the boundary layer
terms written in the ''slow`` variables looks as follows:
\begin{equation}
\label{eps4}
\epsilon^1W^1_\pm(\epsilon^{-1}(x-x^0);\eta)+\epsilon^2W^2_\pm(\epsilon^{-1}
(x-x^0);\eta)=\epsilon^3w_\pm(x;\eta)+O(\epsilon^4|x-x^0|^{-3})\,.
\end{equation} 
In other words, the main regular correction is of order \(\epsilon^3\). 
We write the term \(w_\pm(x)\) as a sum
\begin{equation}\label{www1}w_\pm(x)=w^1_\pm(x)+w^2_\pm(x),\end{equation}                                                  
where   
\begin{eqnarray}
\label{w_1}
w^1_\pm(x)=
-\sum_{j,k=1}^3Q_{jk}\frac{\partial U^0_{\pm}}{\partial
x_j}(x^0)\frac{x_k-x_k^0}{4\pi|x-x^0|^3},
\,
w^2_\pm(x)=\frac{c_{\pm}}{|x-x^0|}.
\end{eqnarray}

Now substituting the asymptotic ansatz (\ref{ansatz2}) for \(U^\epsilon_{\pm}\) 
in the equation (\ref{spctrlprob1}) we get,
using that \(U^0_\pm\) is the solution of unperturbed problem,
\begin{eqnarray*}
-\Delta_x U^\epsilon_{\pm}(x;\eta)&=&
\Lambda^0_{\pm}(\eta)U^0_{\pm}(x;\eta)-
\epsilon^3\Delta_x(\chi(x)w_\pm(x;\eta)+U'_{\pm}(x;\eta))\\
&&+O(\epsilon^4|x-x^0|^{-3}). 
\end{eqnarray*}
On the other hand, since \(\Lambda^\epsilon_\pm(\eta)\) is assumed to be an
eigenvalue, we also have the relationship
\begin{eqnarray*}
-\Delta_x U^\epsilon_{\pm}(x;\eta)&=&
(\Lambda^0_\pm(\eta)+\epsilon^3\Lambda'_\pm(\eta))
(U^0_\pm(x;\eta)+\epsilon^3(\chi(x)w_\pm(x)+U'_{\pm}(x;\eta)))\\
&&+O(\epsilon^4|x-x^0|).
\end{eqnarray*}
Thus, setting the coefficients on \(\epsilon^3\) equal, we obtain the 
following equation:
\begin{eqnarray}
\label{u'}
-\left(\Delta_x+\Lambda^0_{\pm}(\eta)\right)U'_\pm(x;\eta)
=\Lambda'_\pm(\eta)U^0_\pm(x;\eta)\\
+\left(\Delta_x+\Lambda^0_{\pm}(\eta)\right)
\left(\chi(x)w_\pm(x;\eta)\right), \quad x\in \varpi \nonumber
\end{eqnarray}
with the boundary and quasi-periodicity conditions
\begin{eqnarray}
\label{u'_bound}
\partial_n U'_{\pm}(x;\eta)&=&0,\, x\in \gamma,\\
\label{u'_bound_per}
U'_{\pm} (y,0; \eta)&=&e^{-i\eta}U'_{\pm} (y,1;\eta),\, y\in\gamma_0,\\
\partial_z U'_{\pm}(y,0;\eta)&=&e^{-i\eta}\partial_z U'_{\pm} (y,1; \eta),\, y\in\gamma_0.\nonumber
\end{eqnarray}

By the Fredholm alternative the boundary
value problem (\ref{u'})-(\ref{u'_bound_per}) has a solution \(U'_\pm\) 
if and only if the right hand side
of the equation (\ref{u'}) is orthogonal to the eigenfunction \(U^0_\pm\)
in \(L^2(\varpi)\). The normalisation \(\left\|U^0_\pm;L^2(\varpi)\right\|=1\) 
then yields the expression for the correction term
\begin{eqnarray}
\label{l'}
\Lambda'_\pm(\eta)&=&\Lambda'_\pm(\eta)\int_{\varpi}
\overline{U^0_\pm(x;\eta)}U^0_\pm(x;\eta)\,dx\\
&=&-\int_\varpi \overline{U^0_\pm(x;\eta)}
\left(\Delta_x+\Lambda^0_\pm(\eta)\right)
\left(\chi(x)w_\pm(x;\eta)\right)\,dx\nonumber\,.
\end{eqnarray}
Recalling that the functions in (\ref{w_1}) are harmonic, we have 
\[
 (\Delta_x+\Lambda^0_\pm(\eta))(\chi(x)w_\pm(x;\eta))=O(|x-x^0|^{-2}).
\]
Hence the last integral in (\ref{l'}) converges absolutely, and therefore it can be
calculated as a limit of an integral over the set \(\varpi\setminus
B_r(x^0)\) as \(r\to +0\). Since the cut-off
function \(\chi\) vanishes at the boundary of \(\varpi\) we obtain by the Green 
formula that
\begin{equation}
\label{L}
\Lambda'_\pm(\eta)= 
-\lim_{r\to +0}\int_{{\mathbb S}_r}\left(
\overline{U^0_\pm(x;\eta)}\partial_{n(x)} 
w_\pm(x;\eta)-w_\pm(x;\eta)\overline{\partial_{n(x)} U^0_\pm(x;\eta) }\right)\,ds.
\end{equation}

Substituting \(w_\pm\) in (\ref{L}) we get
\[\Lambda'_\pm(\eta)=K_1+K_2+K_3+K_4\]
where
\begin{eqnarray*}
K_j&=&-\lim_{r\to
0}\int_{\mathbb{S}_r(x^0)}\overline{U^0_\pm(x;\eta)}\partial_{n(x)} w^j_\pm(x;\eta)
ds_x,\,j=1,2,\\
K_{j+2}&=&\lim_{r\to 0}\int_{\mathbb{S}_r(x^0)}w^j_\pm(x;\eta)
\overline{\partial_{n(x)} U^0_\pm(x;\eta)}ds_x,\,j=1,2.
\end{eqnarray*}
             
The term \(K_1\) is computed by a direct integration:
\begin{eqnarray*}
K_1&=&\sum_{j,k=1}^3Q_{jk}\frac{\partial U^0_{\pm}}{\partial x_j}(x^0;\eta)
\lim_{r\to +0}
\int_{{\mathbb S}_r(0)} 
\overline{U^0_\pm(x^0+x;\eta)}
\partial_{n(x)}\frac{x_k}{4\pi|x|^3}\,ds_x\\
&=&\sum_{j,k=1}^3Q_{jk}\frac{\partial U^0_{\pm}}{\partial x_j}(x^0;\eta)
\lim_{r\to +0}
\int\limits_{{\mathbb S}_r(0)} 
\overline{U^0_\pm(x^0+x;\eta)}
\left(\frac{2x_k}{4\pi r^4}\right) \,ds_x\\
&=&
\frac{2}{3}(\nabla_xU_\pm^0(x^0;\eta))^\top Q \nabla_x
\overline{U_\pm^0(x^0;\eta)}\,.
\end{eqnarray*}

Since 
\(\partial_{n(x)} w^2_\pm(x;\eta)=-\partial_\rho \left(c_{\pm}\rho^{-1}\right)+O(\rho^{-3})\) 
on \( {\mathbb S}_r(x^0)\) (see (\ref{number}) we have
\begin{multline*}
K_2=\lim_{r\to 0}\int_{{\mathbb S}_r(x^0)}
\overline{U^0_\pm(x;\eta)}\frac{\partial}{\partial\rho}
\left(\frac{c_{\pm}}{|x|}\right) \,d\varphi\\
=-\lim_{r\to 0}\int_{ {\mathbb S}_r(x^0)}
\overline{U^0_\pm(x;\eta)}c_{\pm}r^{-2}d\varphi
=-4\pi c_\pm \overline{U^0_\pm(x^0;\eta)}.
\end{multline*}

For the computation of \(K_3\) we write
\begin{equation*}
w^1_\pm(x)\overline{\partial_{n(x)} U^0_\pm(x;\eta)}
=
\sum_{j,k=1}^3Q_{jk}\frac{\partial U^0_\pm}{\partial x_j}(x^0;\eta)
\frac{x_k-x_k^0}{4\pi|x-x^0|^4} (x-x^0)\overline{\nabla_xU^0_\pm(x;\eta)}\,.
\end{equation*}
This yields us
\begin{eqnarray*}
K_3&=&
\sum_{j,k=1}^3Q_{jk}\frac{\partial U^0_\pm}{\partial x_j}(x^0;\eta)
\lim_{\rho\to 0}\int\limits_{ {\mathbb S}_\rho(0)}
\frac{x_k^2}{4\pi|x|^4}\overline{\frac{\partial U^0_\pm}{\partial
x_k}(x^0+x;\eta)}\,ds_x\\
&=&
\frac{1}{3}\nabla_xU^0_\pm(x^0;\eta)^\top Q
\overline{\nabla_xU^0_\pm(x^0;\eta)}\,.
\end{eqnarray*}
Finally, it is easy to see that
\begin{equation*}
K_4=-\lim_{\rho\to +0}
\int_{{\mathbb S}_\rho(x^0)} 
\frac{c}{|x-x^0|}\overline{\partial_{n(x)} U^0_\pm(x;\eta)}\,ds=0\,.
\end{equation*}
Collecting the previous results we finally obtain the
formula for the main asymptotic correction term
\[\Lambda'_\pm(\eta)=F_{\pm}(\eta)\]
where
\begin{eqnarray*}
\label{maincorr}
F_{\pm}(\eta)&=&
\nabla_xU^0_\pm(x^0;\eta)^\top Q \overline{\nabla_xU^0_\pm(x^0;\eta)}
\\
&+&\left(M_1 +(\eta-\pi\pm\pi)^2\right)\left|U^0_\pm(x^0;\eta)\right|^2\textrm{meas}_3(\theta)\,.
\end{eqnarray*}

The remainder \(\widetilde{\Lambda}^\epsilon_\pm(\eta)\) in the asymptotic
expansion \eqref{ansatz1} will be estimated in \(\S 4\).
Note that \(F_+(\pi)=F_-(\pi)\) and denote it by \(F_0\), so that
\begin{equation}
\label{F_0}
F_0=\nabla_xU^0_\pm(x^0;\pi)^\top Q \overline{\nabla_xU^0_\pm(x^0;\pi)}
+\left(M_1 +\pi^2\right)\left|U^0_\pm(x^0;\pi)\right|^2\textrm{meas}_3(\theta).
\end{equation}

\subsection{Perturbation: \(\eta\approx\pi\)}
\label{subsec:34}
In this section we will study the behaviour of the eigenvalues
\(\Lambda^\epsilon_\pm(\eta)\) when \(\eta\) is close to \(\pi\).  For the unperturbed
problem, i.e. when \(\epsilon=0\), the eigenvalue
\(\Lambda_\pm^0(\pi)=M_1+\pi^2\)
has multiplicity two. Due to this change of the multiplicity our analysis
becomes slightly different. However, we may use the previous calculations.

In the vicinity of \(\eta=\pi\) we apply the formula \eqref{psi}. 
In this case the asymptotic ans\"atze
(\ref{ansatz1}) and (\ref{ansatz2}) are still valid but their entries
depend on the deviation parameter \(\psi\). Moreover the main regular part
of the expansion becomes the linear combination of the eigenfunctions
\(U^0_{+}\) and \(U^0_{-}\):
\begin{eqnarray*}
\mathbf{U}^0_\pm(x;\pi,\psi)
&=&a_1^{\pm}(\psi)U^0_+(x;\pi)+a_2^\pm(\psi)U^0_-(x;\pi)\nonumber\\
&=&\left(a_1^\pm(\psi)e^{\pi i z}+
a_2^\pm(\psi)e^{-\pi i z}\right)V_1(y)
\end{eqnarray*}
where the coefficient vector
\(\mathbf{a}^\pm(\psi)=(a_1^\pm(\psi),a_2^\pm(\psi))^\top\)
is to be determined. Without loss of generality we may assume that 
\(|\mathbf{a}^\pm|=1\). 
By the same argument as in the previous section, we have
the boundary layers \(\mathbf{W}^1\)and \(\mathbf{W}^2\) as 
in (\ref{W1x}) and (\ref{W2x}). However, in this case they 
are the linear combinations 
\begin{eqnarray}\label{WWW}
\mathbf{W}^j(\xi;\pi,\psi)&=&a^\pm_1(\psi)W^j_{+}(\xi;\pi)+a^\pm_2(\psi)W^j_{-}(\xi;\pi),\,j=1,2.
\end{eqnarray}
Then equation (\ref{u'}) takes the form
\begin{eqnarray}\label{u'psi}
(-\Delta_x&-&\Lambda^0(\pi))\mathbf{U}'_\pm(x;\pi,\psi)=
\Lambda'(\pi,\psi)\mathbf{U}^0_\pm(x;\pi,\psi)\nonumber\\
&-&\left(-\Delta_x- (M_1+\pi^2)\right)\left(\chi(x)\mathbf{w}(x;\pi,\psi)\right),\
x\in\varpi\setminus x^0,
\nonumber
\end{eqnarray}
where the \(\mathbf{w}(x;\pi,\psi)\) is calculated for the boundary layer terms (\ref{WWW}) according to the 
formulae (\ref{eps4}) and (\ref{www1}).
The boundary and quasi-periodicity conditions (\ref{u'_bound}) and
(\ref{u'_bound_per})
turn into following:
\begin{eqnarray}
\label{u'psi_bound}
\partial_n \mathbf{U}'_{\pm}(x;\pi,\psi)&=&0, \quad x\in \gamma,\\
\label{u'psi_bound_per}
\mathbf{U}'_{\pm} (y,0;\pi,\psi)&=&e^{-i\pi}\mathbf{U}'_{\pm} (y,1;\pi,\psi)-i\psi e^{-i\pi}
\mathbf{U}^0_\pm(y,1;\pi,\psi),  
y\in\gamma_0 \\
\partial_z \mathbf{U}'_{\pm} (y,0;\pi,\psi)&=&
e^{-i\pi}\partial_z \mathbf{U}'_{\pm} (y,1;\pi,\psi)-i\psi e^{-i \pi}\partial_z
\mathbf{U}^0_\pm(y,1;\pi,\psi), 
 y\in\gamma_0.\nonumber
\end{eqnarray}
Note that the extra terms on the 
right-hand side of \eqref{u'psi_bound_per} is due to 
the original quasi-periodicity conditions \eqref{percond}
and the Taylor formula
\[e^{-i(\pi+\psi\epsilon^3)}=e^{-i\pi}(1-i\psi\epsilon^3+O(\epsilon^6)),\,
\epsilon\to 0.\]
To find \(\Lambda'(\pi,\psi)\) we multiply equation \eqref{u'psi} by
\(\overline{U^0_+(x;\pi)}\)
and integrate over \(\varpi\)
\begin{eqnarray}
\label{59}
\int_\varpi
\overline{U^0_+(x;\pi)}(-\Delta_x&-&\Lambda^0(\pi))\mathbf{U}'_\pm(x;\pi,\psi)\,dx=
\Lambda'(\pi,\psi)a_1^\pm (\psi) \\
&-&\int_\varpi\overline{U^0_+(x;\pi)}\left(-\Delta_x-
(M_1+\pi^2)\right)(\chi(x)\mathbf{w}(x;\pi,\psi))\,dx.\nonumber
\end{eqnarray}
The integral on the right can be calculated in a similar way as in the previous
section and it will provide the formula
\begin{eqnarray}
&-&\int_\varpi \overline{U^0_+(x;\pi)}\left(-\Delta_x-
(M_1+\pi^2)\right)\left(\chi(x)\mathbf{w}(x;\pi,\psi)\right)\,dx
=a_1^\pm(\psi)F_0+ \nonumber\\
&+&
a_2^\pm(\psi)\left(\nabla_xU^0_-(x^0;\pi)^\top Q\overline{\nabla_xU^0_+(x^0;\pi)}
-(M_1+\pi^2)U^0_-(x^0;\pi)\overline{U^0_+(x^0;\pi)}\right)\nonumber\\
&=&a_1^\pm(\psi)F_0+a_2^\pm(\psi)F_1,\nonumber
\end{eqnarray}
where \(F_0\) is given in \eqref{F_0} and
\begin{eqnarray}
\label{F_1}
F_1&=&\nabla_xU^0_-(x^0;\pi)^\top Q\overline{\nabla_xU^0_+(x^0;\pi)}\\
&+&(M_1+\pi^2)U^0_-(x^0;\pi)\overline{U^0_+(x^0;\pi)}\textrm{meas}_3(\theta).\nonumber
\end{eqnarray}
The integral on the left hand side of \eqref{59} is calculated by means of the Green
formula
\begin{eqnarray*}
&&\int_\varpi
\overline{U^0_+(x;\pi)}(-\Delta_x-\Lambda^0(\pi))\mathbf{U}'_\pm(x;\pi,\psi)\,dx\nonumber \\
&&=\int_{\partial\varpi}(\mathbf{U}'_\pm(x;\pi,\psi)\overline{\partial_n U^0_{+}(x;\pi)}-
\partial_n \mathbf{U}'_\pm(x;\pi,\psi)\overline{U^0_{+}(x;\pi)})\,ds_x=:I\,.
\end{eqnarray*}
For calculating this surface integral we take into account the Neumann conditions
\eqref{u'psi_bound}
for \(\mathbf{U}'_\pm\) and \eqref{form6} for \(U^0_+\) on \(\gamma\),
the quasi-periodicity conditions \eqref{form7} and \eqref{u'psi_bound_per} to obtain
\begin{eqnarray}
I=&&\int_\gamma \big(-\overline{\partial_z U^0_+(y,0;\pi)} \mathbf{U}'_\pm(y,0;\pi,\psi)
+\overline{U^0_+(y,0;\pi)}\partial_z \mathbf{U}'_\pm(y,0;\pi,\psi)\big) dy \nonumber\\
+&&\int_\gamma \big(
\overline{\partial_z U^0_+(y,1;\pi)} \mathbf{U}'_\pm(y,1;\pi,\psi)-
\overline{U^0_+(y,0;\pi)}\partial_z \mathbf{U}'_\pm(y,1;\pi,\psi)\big) dy \nonumber\\
=&& 2\pi\psi a_1^\pm(\psi) \int_\gamma |V_1(y)|^2\,dy =2\pi\psi
a_1^\pm(\psi).\nonumber
\end{eqnarray}
In last equality we have used the normalisation condition for \(V_1\) and the trivial
equality
\(\|U^0_+;L^2(\varpi)\|=\|V_1;L^2(\gamma)\|\).

Thus the first equation for \(\Lambda'(\pi,\psi)\) takes form
\begin{equation*}
\label{Lambda'-first}
a_1^\pm(\psi)(F_0-2\pi\psi)+a_2^\pm(\psi)F_1=\Lambda'(\pi,\psi)a_1^\pm(\psi).
\end{equation*}
Applying the same procedure for $U^0_-$, we get the second equation  
\begin{equation*}
\label{Lambda'-second}
-a_1^\pm(\psi)\overline{F_1}
+a_2^\pm(\psi)(F_0+2\pi\psi)=\Lambda'(\pi,\psi)a_2^\pm(\psi).
\end{equation*}
Finally, we have noticed that the coefficients \(a_1^\pm(\psi)\) and \(a_2^\pm(\psi)\) satisfy the
system of algebraic equations written in the matrix form
\begin{equation}
\label{sys a}
\left(\begin{matrix}
-2\pi\psi+F_0& F_1\\
-\overline{F_1}&2\pi \psi+F_0
\end{matrix}
\right)
\left(
\begin{matrix}
a_1^\pm(\psi)\\a_2^\pm(\psi)
\end{matrix}
\right)
=\Lambda'_\pm(\pi;\psi)
\left(
\begin{matrix}
a_1^\pm(\psi)\\a_2^\pm(\psi)
\end{matrix}
\right) .
\end{equation}
In other words, the main asymptotic corrections to the eigenvalues
\(\Lambda^0_\pm(\eta)\)
are eigenvalues of the matrix on the left-hand side of (\ref{sys a}).
At the end a simple computation provides us the formula
\begin{equation}
\label{L'psi}
\Lambda'_\pm(\pi,\psi)=F_0 \pm\sqrt{4\pi^2\psi^2+|F_1|^2}\,,
\end{equation}
where \(F_0\) and \(F_1\) are given in \eqref{F_0} and \eqref{F_1}.
We emphasise that in the case \(F_1\ne 0\) the graphs of the eigenvalues
\(\Lambda^\epsilon_\pm(\pi+\psi\epsilon^3)\) given by \eqref{ansatz1}
and \eqref{L'psi} look like the ones in fig. 5c.

\subsection{Final formulas}
\label{subsec:35}

For the Neumann problem (\ref{form1})-(\ref{form2}) we have \(M_1=0\), 
\(V_1(y)=\textrm{meas}_2(\omega)^{-\frac12}\) and
\[U^0_{\pm}(x;\pi)=e^{\pm i\pi z}(\textrm{meas}_2(\omega))^{-\frac12}.\]
Substituting these into the formulae (\ref{F_0}) and (\ref{F_1}) we obtain
\begin{equation*}
F_0=\pi^2 \frac{Q_{33}+\textrm{meas}_3(\theta)}{\textrm{meas}_2(\omega)},\,
|F_1|=\pi^2 \frac{-Q_{33}+\textrm{meas}_3(\theta)}{\textrm{meas}_2(\omega)}.
\end{equation*}
By Remark 3.1 it follows that \(-Q_{33}>\textrm{meas}_3(\theta)\) and hence
\begin{equation}
 \label{oulu1}
F_0<0,\,|F_1|>0.
\end{equation}

However, for the mixed boundary value problem (\ref{dir-1})
\[
U^0_{\pm}(x;\pi)=e^{\pm i\pi z}V_1(y),
\]
where \(\{M_1,V_1\}\) is the first eigenpair of the Dirichlet problem (\ref{S4}). 
In this case the coefficients \(F_0\) and \(F_1\) are
\begin{eqnarray}\label{oulu2}
F_0&=&\nabla_y V_1(y^0)^\top Q'\nabla_y V_1(y^0)+(M_1+\pi^2(Q_{33}+\textrm{meas}_3(\theta))|V_1(y^0)|^2\\
F_1&=&e^{-2\pi i z_0}
\big[\nabla_y V_1(y^0)^\top Q'\nabla_y V_1(y^0)+
(M_1-\pi^2(Q_{33}-\textrm{meas}_3(\theta))|V_1(y^0)|^2 \big]\nonumber\\
&-&i2\pi e^{-2\pi i z_0}\{Q_{31}\frac{\partial V_1}{\partial y_1}(y^0)+
Q_{32}\frac{\partial V_1}{\partial y_2}(y^0)\}V_1(y^0)\nonumber
\end{eqnarray}
where \(Q'\) is the \(2\times 2\) upper left block of the 
virtual mass matrix \(Q\) for the set \(\overline{\theta}\subset
\mathbb{R}^3\).

\section{Justification of the asymptotics and existence of a spectral gap}
\label{sec:4}

In this section we will first estimate the difference
\(\Lambda_\pm^\epsilon(\eta)-\Lambda_\pm^0(\eta)\) 
outside the \(C\epsilon^{\frac72}\)-neighbourhood of the point \(\eta=\pi\) and
show that the perturbed eigenvalues \(\Lambda_\pm^\epsilon(\eta)\) are close to
the eigenvalues of the unperturbed problem. This is a rather standard step in 
the asymptotic analysis of eigenvalues in singularly perturbed domains. 
The second step is more elaborated.

In order to prove the opening of the spectral gap we have to derive approriate estimates of the remainder
in the asymptotic expansion of \(\Lambda_\pm^\epsilon(\eta)\) constructed in section 3
inside the \(C\epsilon^{\frac72}\)-neighbourhood of the point \(\eta=\pi\). 
To this end, apply the following classical
lemma on ''near eigenvalues and eigenvectors`` (see
\cite{ViLu} and also, e.g., \cite[Ch. 5]{BiSo}).

\begin{lemma}
\label{nearEig}
Let \({\cal B}\) be a selfadjoint, positive, and compact operator in Hilbert
space \({\cal H}\) with the inner product \((\cdot,\cdot)_{\cal H}\). If there
exists a number \(b>0\) and an element \({\cal W}\in {\cal H}\) such that
\(\|{\cal W}\|_{\cal H}=1\) and \(\|{\cal BW}-b{\cal W}; {\cal H}
\|=\tau\in(0,b)\), then the segment \([b-\tau,b+\tau]\) contains at least one
eigenvalue of \({\cal B}\).
\end{lemma}

\subsection{\(\eta\) far from \(\pi\)}
\label{subsec:41}

In this section we will prove the following theorem
\begin{theorem}
For any \(\alpha> 0\) there exist \(C,\,C(\alpha)>0\) and \(\epsilon_\alpha>0\), such
that for all \(\epsilon\in[0,\epsilon_\alpha)\)
\begin{equation*}
\label{NN}
\Lambda_k^0(\eta)-C_\alpha\epsilon^{3-\alpha}\leq\Lambda_k^\epsilon(\eta)\leq
\Lambda_k^0(\eta)+C\epsilon^3 ,\,k=1,2.
\end{equation*}
\end{theorem}

The following lemma is a direct
consequence of the one-dimensional Hardy inequality:
if \(u(b)=0\), then
\begin{equation*}
\label{hardy}
\int_a^b |u(r)|^2\,dr\leq 4\int_a^b  |\partial_r u(r)|^2r^2\,dr \,.
\end{equation*}

\begin{lemma}
There exists a positive constant \(C=C(\theta, \varpi,x^0)\) such that the
following inequality holds:
\[
\left\|(r+\epsilon)^{-1}u;L^2(\varpi_\epsilon)\right\|^2\leq
C\left(\left\|\nabla_x
u;L^2(\varpi_\epsilon)\right\|^2
+\|u;L^2(\varpi\setminus\mathcal{U})\|^2\right)\,,
\]
where $r=|x-x_0|$ and \(\mathcal{U}\subset \varpi \) is a fixed
neighbourhood of the point \(x^0\), independent of \(\epsilon\).
\end{lemma}

We also need an a priori estimate for the eigenfunctions
of the problem \eqref{spctrlprob1}-\eqref{percond} in the singularly
perturbed domain \(\varpi_\epsilon\). The method developed in \cite{MaNaPl} gives
such estimates in weighted Kondratiev norms
\begin{equation}
\label{norm1}
\left\|u^\epsilon;V_{\beta}^l(\varpi_\epsilon)\right\|^2=
\sum_{k=0}^l
\int_{\varpi_\epsilon}r^{2(\beta-l+k)}
\left|\nabla_x^ku^\epsilon(x)\right|^2\, dx\,.
\end{equation}
and the step-weighted norms
\begin{equation}
\label{norm2} 
\left\|U^\epsilon;V_{\beta,0}^2(\varpi_\epsilon)\right\|=
\left(\|\nabla_xU^\epsilon; V^1_\beta(\varpi_\epsilon)\|^2+
\|U^\epsilon;V^0_{\beta-1}(\varpi_\epsilon)\|^2\right)^{1/2}.
\end{equation}
For the Neumann problem, the latter norm provides asymptotically
sharp estimates and we formulate this estimate with references
to the general result in \cite[Ch. 4 and 5]{MaNaPl} and paper \cite{na235}
where the detailed explanation of the method is given for a concrete 
boundary value problem.

\begin{lemma} 
Let \(U^\epsilon\) be an eigenfunction of problem
\eqref{spctrlprob1}-\eqref{percond} with the eigenvalue \(\Lambda^\epsilon\).
For any \(\beta\in (-\frac{1}{2},\frac12)\) the inequality
\[
\left\|U^\epsilon; V^2_{\beta,0}(\varpi_\epsilon)\right\|\leq
C_\beta\Lambda_\epsilon
\left\|U^\epsilon;L^2(\varpi_\epsilon)\right\|
\] 
is valid while a constant \(C_\beta\) is independent of \(U^\epsilon\)
and \(\epsilon\in (0,\epsilon_0],\, \epsilon_0>0\). At the same time 
\(C_\beta\to +\infty\) as \(\beta\to \pm \frac{1}{2}\).
\end{lemma}

\textbf{Proof of Theorem 4.2}:

According to the max-min principle (cf. \cite[Thm. 10.2.2]{BiSo})
\begin{equation}
\Lambda_k^\epsilon(\eta)=\sup_{E_k} \inf_{u\in E_k}\frac{\|\nabla_x
u;L^2(\varpi_\epsilon)\|^2}
{\|u;L^2(\varpi_\epsilon)\|^2}\,,
\end{equation}
where supremum is taken over all subspaces 
\(E_k\subset H^1_{\eta}(\varpi_\epsilon)\) with co-dimension \(k-1\).
Every such \(E_k\) contains the linear combination of 
eigenfunctions \(v=a_1U_1^0+a_2U_2^0+\ldots+ a_k U_k^0\) with \(|a_1|^2+\cdot+|a_k|^2=1\), so in the case
\(c_k\epsilon^3\leq 1\) we conclude
\begin{eqnarray*}
\Lambda_k^\epsilon(\eta)&&\leq
\frac{\|\nabla_x v;L^2(\varpi_\epsilon)\|^2}
{\|v;L^2(\varpi_\epsilon)\|^2}\leq
\frac{\|\nabla_x v;L^2(\varpi)\|^2}
{\|v;L^2(\varpi)\|^2-\|v;L^2(\theta_\epsilon)\|^2}
\\
&&\leq\frac{
|a_1|^2\Lambda^0_1(\eta)+|a_2|^2\Lambda^0_2(\eta)+\ldots+|a_k|^2\Lambda^0_k(\eta)}
{|a_1|^2+|a_2|^2+\ldots+|a_k|^2-c_k\epsilon^3}\leq
\frac{\Lambda_k^0(\eta)}{1-c_k\epsilon^3}\\
&\leq& \Lambda^0_k(\eta)(1+C_k\epsilon^3).                                         
\end{eqnarray*}

Next we give a proof of the first inequality of the statement for
\(\Lambda^\epsilon_1\). By the minimum principle \cite[Thm. 10.2.1]{BiSo}, we have
\[
\Lambda_1^0(\eta)=
\inf_{v\in H^1_{\eta}(\varpi)}
\frac{\left\|\nabla_x v;L^2(\varpi)\right\|^2 }{\left\|v;L^2(\varpi)\right\|^2}.
\]
Replacing $v$ by the eigenfunction \(u_1^\epsilon\) normalised to 1 in
\(L^2(\varpi_\epsilon)\)  is not possible and we construct an extension
of function $u_1^\epsilon$ as follows. 
Consider the function \(w(\epsilon^{-1}(x-x^0))=u_1^\epsilon(x)\) which can be
written as
\[w(\xi)=w^\perp(\xi)+w_0,\]
where \(w^\perp$ is orthogonal to 1 in \(L^2(\partial \theta)\). Since
\(\partial\theta\) is smooth,
function \(w^\perp\) can be extended to a function 
\(\widetilde{w}^\perp\in H^1(\theta)\) by some fixed 
continuous linear extension operator such that the inequality
\begin{eqnarray}\label{ww}
\|\nabla_\xi\widetilde{w}^\perp;L^2(\theta)\|&\leq&C\|w^\perp;H^1(B_R\setminus\theta)\|\\
&\leq&C\|\nabla_\xi w^\perp;L^2(B_R\setminus\theta)\|=C\|\nabla_\xi w;L^2(B_R\setminus\theta)\|\nonumber
\end{eqnarray}
is valid. Setting
\(\widetilde{u}_1^\epsilon(x):=\widetilde{w}^\perp(\epsilon^{-1}(x-x^0))+w_0\) we derive from (\ref{ww})
\begin{equation*}\label{uu}
\|\nabla_x\widetilde{u}^\epsilon_1;L^2(\theta_\epsilon)\|\leq C 
\|\nabla_x u^\epsilon_1;L^2(B_{R\epsilon}\setminus\theta_\epsilon)\|.
\end{equation*}
Now we have
\begin{eqnarray*}\label{inequality}
\Lambda_1^0(\eta)&\leq& 
\frac{\left\|\nabla_x u_1^\epsilon;L^2(\varpi_\epsilon)\right\|^2+
\left\|\nabla_x
\widetilde{u}_1^\epsilon;L^2(\theta_\epsilon)\right\|^2}
{\left\|u_1^\epsilon;L^2(\varpi_\epsilon)\right\|^2+
\left\|\widetilde{u}_1^\epsilon;L^2(\theta_\epsilon)\right\|^2}\nonumber\\
&\leq&{\Lambda_1^\epsilon(\eta)+
\left\|\nabla_x\widetilde{u}_1^\epsilon;L^2(\theta_\epsilon)\right\|^2}\\
&\leq& 
\Lambda_1^\epsilon(\eta) + C \left\|\nabla_x 
u_1^\epsilon;L^2(B_{R\epsilon}\setminus\theta_\epsilon)\right\|^2\nonumber
\end{eqnarray*}
Owing to Lemma 4.4 and the definitions (\ref{norm1}) and (\ref{norm2}) we have    
\begin{eqnarray*}
\|\nabla_x u^\epsilon_1;L^2(B_{R\epsilon)}\setminus\theta_\epsilon)\|^2&\leq&
c\epsilon^{2(1-\beta)}\|\tau^{\beta-1}\nabla_x u^\epsilon_1;L^2(B_{R\epsilon}\setminus\theta_\epsilon)\|^2\\
&\leq& c\epsilon^{2(1-\beta)}\|u^\epsilon_1;V_{\beta,0}^2(\varpi_\epsilon)\|^2,
\end{eqnarray*}
where we can take \(\beta=-\frac12+\frac\alpha2\in(-\frac12,\frac12)\) and obtain the first estimate. The
second inequality for the eigenvalues \(\Lambda_k^\epsilon(\eta)\) is proved in
a similar manner but with an application of the max-min principle.

\subsection{\(\eta\approx\pi\)}
\label{subseq:42} 

Theorem 4.2 ensures the following statement: For any
small positive \(\eta_0\) and \(m_0\) one can find \(\epsilon_0=\epsilon(\eta_0,m_0)>0\)
such that for all \(\epsilon\in [0,\epsilon_0)\) and
\(\eta\in(\pi-\eta_0,\pi+\eta_0)\) the interval
\((M_1+(\pi-\eta_0)^2-m_0,M_1+(\pi-\eta_0)^2+m_0)\) contains either two simple 
eigenvalues, or one eigenvalue with multiplicity two.

The immediate objective becomes to justify the asymptotics of 
the eigenvalues constructed in Section 3.

In Hilbert space \({\cal H}^\epsilon(\eta):=H^1_\eta(\varpi_\epsilon)\) 
with scalar product
\[
(u,v)_{H^1_\eta(\varpi_\epsilon)}=(\nabla_x u,\nabla_x v)_{\varpi_\epsilon}
\]
we define the operator
\(\mathcal{B}^\epsilon(\eta)\) by the formula
\begin{equation}\label{BB}
(\mathcal{B}^\epsilon(\eta)u,v)_{{\cal H}^\epsilon(\eta)}=(u,v)_{
\varpi_\epsilon},\, u,v\in {\cal H}^\epsilon(\eta)\,.
\end{equation}
The operator is self-adjoint, positive and compact. Thus the spectrum
of \(\mathcal{B}^\epsilon(\eta)\) consists of the point
$\beta_0=0$ implying the essential spectrum and of the positive decreasing sequence of
eigenvalues
\begin{equation}\label{seq}
\beta_1^\epsilon(\eta)\geq \beta_2^\epsilon(\eta)\geq \ldots\geq
\beta_j^\epsilon(\eta)\geq\ldots
\end{equation} 
The variational formulation of the spectral problem (\ref{spctrlprob1})
\begin{equation*}
(\nabla U^\epsilon,\nabla V)_{\varpi^\epsilon}=\Lambda^\epsilon(\eta)(U^\epsilon,
V)_{\varpi_\epsilon},\,
V\in {\cal H}^\epsilon(\eta),
\end{equation*}
gives the following relationship between the eigenvalues of the operator
\(\mathcal{B}^\epsilon(\eta)\) and the boundary value problem
(\ref{spctrlprob1})-(\ref{percond}) 
\begin{equation}\label{beta}
\Lambda_j^\epsilon(\eta)=\beta_j^\epsilon(\eta)^{-1}\,.
\end{equation}
The relationship (\ref{beta}) turns the sequence (\ref{seq}) into the sequence (\ref{seq-Lambda}).

We now apply Lemma 4.1 to the operator (\ref{BB}). To this end we set 
\begin{eqnarray}\label{AA2}
b_{\pm}&=&(\ell_\pm)^{-1}=(\Lambda_\pm^0(\pi)+\epsilon^3\Lambda'_{\pm}(\pi,
\psi))^{-1},\\
\mathcal{V}_\pm(x)&=&\|\mathcal{U}_\pm;
{\cal H}^\epsilon(\pi+\psi\epsilon^3)\|^{-1}\mathcal{U}_\pm(x),\nonumber
\end{eqnarray}
where
\begin{eqnarray}\label{AA1}
\mathcal{U}_\pm(x)&=&X_\epsilon(x)\mathbf{U}^0_\pm(x;\pi,\psi)+
(1-X_\epsilon(x))\mathbf{P}^0_\pm(x^0;\pi,\psi) \\
&+&\chi_\theta(x)\mathcal{W}_\pm(x;\pi,\psi)+\epsilon^3X_\epsilon(x)\mathbf{U}'_\pm(x;\pi,\psi)
+\epsilon^6 \mathbf{U}''(x;\psi),\nonumber
\end{eqnarray}
and
\begin{eqnarray*}
\mathbf{P}^0_\pm(x;\pi,\psi)&=&\mathbf{U}^0_\pm(x^0;\pi,\psi)+
\sum_{j=1}^3x_j\frac{\partial \mathbf{U}^0_\pm}{\partial x_j}(x^0;\pi,\psi)\\
&+&\frac{1}{2}\sum_{j,k=1}^3x_jx_k\frac{\partial^2 \mathbf{U}^0_\pm}{\partial
x_jx_k}(x^0;\pi,\psi),\\
\mathcal{W}_\pm(x;\pi,\psi)&=&\epsilon
\mathbf{W}^1_\pm(\epsilon^{-1}(x-x^0);\pi,\psi)+\epsilon^2\mathbf{W}^2_\pm(\epsilon^{-1}
(x-x^0);\pi,\psi)\,.
\end{eqnarray*}
Here the cut-off function \(X_\epsilon\) is defined by
\[X_\epsilon(x)=1-\chi_\theta(\epsilon^{-1}(x-x^0))\] 
where \(\chi_\theta\in C_0^\infty(\mathbb{R}^3)\) 
is such that $\chi_\theta=1$ in the neighbourhood of the set
\(\overline{\theta}\). The term \(\mathbf{U}''\in H^1(\varpi)\) in the sum is added 
to satisfy the quasi-periodicity condition (\ref{percond}) with \(\eta=\pi^2+\epsilon^3\psi\). 
Due to the definition of the functions $W^1_\pm$, $W^2_\pm$, \(X_\epsilon\),
\(U''\) and
\(\chi_\theta\) the function \(\mathcal{V}_\pm\) belongs to space
\(\mathcal{H}^\epsilon(\pi+\psi\epsilon^3)\).

To use Lemma 4.1 we have to verify that
\(b_\pm=(\ell_\pm)^{-1}\) and
\(\left\|U^0_\pm;{\mathcal{H}^\epsilon(\pi+\psi\epsilon^3)}\right\|\)
are separated from zero, and then we have to estimate 
\[\tau_\pm=\left\|\mathcal{BW}_\pm-b\mathcal{W}_\pm;
{\mathcal{H}^\epsilon(\pi+\psi\epsilon^3)}\right\| \]
for \(\epsilon>0\) small enough. We proceed with the following assertion 
 
\begin{lemma}
For all \(\epsilon>0\) there holds the inequality
\begin{equation}
\label{U}
\left|\left\|\mathcal{U}_\pm;{\mathcal{H}^\epsilon(\pi+\psi\epsilon^3)}
\right\|
-(M_1+\pi^2)^{1/2}|\mathbf{a}^\pm|_{{\mathbb C}^2}\right|\leq 
c\epsilon^{3/2}(1+\epsilon^{3/2}\psi).
\end{equation}
\end{lemma}

\begin{proof} Since 
\(\mathbf{U}^0_\pm(x;\pi,\psi)=a_1^\pm(\psi) U^0_{+}(x)+a_2^\pm(\psi) U^0_{-}(x)\), 
\(|\mathbf{a}^\pm|_{\mathbb{C}^2}=1\), is
a solution of the problem (\ref{form5})-(\ref{form7}) for
\(\eta=\pi+\psi\epsilon^3\), we have
\begin{equation*}
\left\|\nabla_x \mathbf{U}_\pm^0;L^2(\varpi)\right\|^2
=(M_1+\pi^2).
\end{equation*}
The expression on left-hand side of the inequality (\ref{U}) is estimated by the
sum of the expressions
\begin{eqnarray*}
&&\left\|\nabla_x \mathbf{U}_\pm^0;{L^2(\theta_\epsilon)}\right\|,\\
&&\left\|(1-X_\epsilon)(\mathbf{U}^0_\pm-\mathbf{P}^0_\pm;
{{\cal
H}^\epsilon(\pi+\psi\epsilon^3)}\right\|, 
\left\|\chi_\theta \mathbf{W}_\pm;{{\cal H}^\epsilon(\pi+\psi\epsilon^3)}\right\|, \\
&&
\left\|\epsilon^3X_\epsilon \mathbf{U}'_\pm;{{\cal
H}^\epsilon(\pi+\psi\epsilon^3)}\right\|,
\left\|\epsilon^6\mathbf{U}'';{{\cal H}^\epsilon(\pi+\psi\epsilon^3)}\right\|.
\end{eqnarray*}                                              

To process the first two norms it is sufficient to observe 
that the supports of the functions 
\(\nabla_x X_\epsilon\) and \(1-X_\epsilon\) are contained in a ball of radius
\(O(\epsilon)\), where the Taylor formula
\begin{equation}\label{Taylor}
\left|\mathbf{U}^0_\pm(x;\pi,\psi)-\mathbf{P}^0_\pm(x;\pi,\psi)\right|\leq
c\left|x-x^0\right|^3\leq C\epsilon^3\,
\end{equation}
is valid.      

For the estimation of the norm 
\(\left\|\chi_\theta \mathbf{W}_\pm;{{\cal H}^\epsilon(\pi+\psi\epsilon^3)}\right\|\)  
we have to take into account the behaviour of the function $\mathbf{W}_\pm$ at infinity. 
A direct calculations leads to the estimate
\begin{eqnarray*}
\left\|\chi_\theta \mathbf{W}_\pm;{\cal H}^\epsilon(\pi+\psi \epsilon^3)\right\|^2& \leq&
c\epsilon^6 \left(\int_{c\epsilon}^C (r^{-6}+(1+|\psi|\epsilon^3)^2
r^{-4})r^2\,dr \right)
|\mathbf{a}^\pm|^2\leq\\ 
&\leq& c\epsilon^{3}(1+\epsilon^{3/2}|\psi|)^2.
\end{eqnarray*}
Finally, using that \(\mathbf{U}'_\pm,\mathbf{U}''\in H^1(\varpi)\), we get the bound
for the last two norms which proves the statement.
\end{proof}

Under the restriction \(\epsilon^{3/2}|\psi|\in(0,\psi_0]\) we conclude
inequalities
\begin{equation*}
\ell_\pm\geq c_\ell>0, \qquad 
\left\|{\cal U}_\pm;{{\cal H}^\epsilon(\pi+\psi\epsilon^3)}\right\|\geq c_{\cal
U}>0\,,
\end{equation*}
with some positive constants \(c_\ell\) and \(c_{\cal U}\), depending on
\(\psi_0\), \(\varpi\) and \(\theta\) only.

For the estimation of \(\tau\) we use the formula for the norm
\begin{eqnarray*}
\tau&=&\left\|\mathcal{B}^\epsilon(\pi+\psi\epsilon^3)\mathcal{V}_\pm-
\ell_\pm\mathcal{V}_\pm;\mathcal{H}^\epsilon(\pi+\psi\epsilon^3)\right\|\\
&=& \max_{Z}\left|\left({\cal B}^\epsilon(\pi+\psi\epsilon^3)
\mathcal{V}_\pm-\ell_\pm\mathcal{V}_\pm, 
Z\right)_{{\cal H}^\epsilon(\pi+\psi\epsilon^3)}\right|\\
&\leq& c_\ell^{-1}c_{\mathcal{U}}^{-1} \max|T_\pm(Z)|\,.
\end{eqnarray*}
Here the maximum is calculated over all functions 
\(Z\in {\cal H}^\epsilon(\pi+\psi\epsilon^3)\)
with the unit norm, and \(T_\pm\) stands for
\begin{eqnarray*}
\label{T}
T_\pm(Z)&=&\left(\nabla_x {\cal U}_\pm, \nabla_x Z\right)_{L^2(\varpi_\epsilon)^3}\\
&-&(\Lambda_\pm^0+\epsilon^3\Lambda'_\pm(\pi,\psi))({\cal
U}_\pm,Z)_{L^2(\varpi_\epsilon)}.\nonumber
\end{eqnarray*}
\(T_\pm(Z)\) can be now represented as a sum of the expressions
\begin{eqnarray}\label{ZZZ}
S^{(1)}_\pm(Z)&=&(-(\Delta_x+\Lambda^0_\pm)\mathbf{U}^0_\pm,X_\epsilon Z)_{\varpi_\epsilon}=0,\nonumber\\
S_\pm^{(2)}(Z)&=&(-[\Delta_x,X_\epsilon](\mathbf{U}^0_\pm-\mathbf{P}^0_\pm),Z)_{\varpi_\epsilon},\nonumber\\
S_\pm^{(3)}(Z)&=&\epsilon(-\Delta_x \mathbf{W}^1_\pm,\chi_\theta Z)_{\varpi_\epsilon}-
(\epsilon\partial_n \mathbf{W}^1_\pm-\nabla_x\mathbf{U}^0_\pm(x^0;\pi,\psi)^\top n,Z )_{\partial\theta_\epsilon}\\
&+&\epsilon^2(-\Delta_x\mathbf{W}^2_\pm,\chi_\theta Z)_{\varpi_\epsilon}-
(\epsilon^2\partial_n \mathbf{W}^2_\pm,Z)_{\partial\theta_\epsilon}\nonumber\\
&-&\frac12\big(\sum_{j,k=1}^3\partial_n(x_jx_k)
\frac{\partial^2 \mathbf{U}^0_\pm}{\partial x_j\partial x_k}(x^0;\pi,\psi),Z\big)_{\partial\theta_\epsilon}\nonumber\\
S_\pm^{(4)}(Z)&=&(-[\Delta_x,\chi_\theta](\epsilon \mathbf{W}^1_\pm+
\epsilon^2\mathbf{W}^2_\pm),Z)_{\varpi_\epsilon}\nonumber\\
&-&(\Lambda^0_\pm+
\epsilon^3\Lambda'_\pm)(\epsilon \mathbf{W}^1_\pm+
\epsilon^2\mathbf{W}^2_\pm,\chi_\theta Z)_{\varpi_\epsilon}\nonumber\\
S^{(5)}_\pm(Z)&=&\epsilon^3 (-(\Delta_x+\Lambda^0_\pm+\Lambda'_\pm)\mathbf{U}'_\pm+
\mathcal{F}_1,Z)_{\varpi\epsilon}\nonumber\\
S^{(6)}_\pm(Z)&=&\epsilon^3(-[\Delta_x,X_\epsilon]\mathbf{U}'_\pm,Z)_{\varpi_\epsilon}
-\epsilon^6\Lambda'_\pm(\mathbf{U}'_\pm,X_\epsilon Z)_{\varpi_\epsilon}
+\epsilon^6(\nabla_x\mathbf{U}'',\nabla_xZ)_{\varpi_\epsilon}\nonumber\\
&-&\epsilon^6(\Lambda^0_\pm+\epsilon^3\Lambda'_\pm)(\mathbf{U}'',Z)_{\varpi_\epsilon}\nonumber
\end{eqnarray}

Here \(\mathcal{F}_1\)  denotes 
the last term in equation (\ref{u'psi}). 
The expressions (\ref{ZZZ}) can be
divided into three types:
\begin{itemize}
\item the terms which appears due to the multiplication the
function \(U^0_\pm\) by the cut-off function,
\item the terms which stem from the asymptotic boundary layers and
\item finally, from terms which are asymptotic smooth corrections of \(U'_\pm\).
\end{itemize}

To estimate the first type of terms it is enough to observe that the supports of
\(\nabla_x
X_\epsilon\) and \(1-X_\epsilon\) are included in a ball of radius
\(O(\epsilon)\),
and to apply the Taylor formula (\ref{Taylor}) in this ball. Hence the H\"older
inequality gives us the upper bound 
\begin{equation*}
c\epsilon^{7/2}\|Z;{H^1(\varpi_\epsilon)}\|
\end{equation*}
for first type terms.

Components of the third type, except for the surface integrals, are evaluated in the
same manner owing to the additional multiplier \(\epsilon^3\).  
The surface integral
\(\left(\partial_n \mathbf{U}'_\pm,Z\right)_{\partial \theta_\epsilon}\) 
can be estimated by the trace inequality and the asymptotic behaviour of the function
\(\mathbf{U}'_\pm\) in the vicinity of the point \(x^0\) 
which provides us the point wise estimate \(|\mathbf{U}'_\pm(x)|\leq c\log|x-x^0|\).

Now it is suffices to estimate the contribution coming from the asymptotic boundary
layer terms. For this assessment we use the main parts of the functions
\(\mathbf{W}^1_\pm\) and \(\mathbf{W}^2_\pm\). Let us denote by 
\[\epsilon \mathbf{W}^1_\pm(\xi;\pi,\psi)+\epsilon^2\mathbf{W}^2_\pm (\xi;\pi,\psi)=\epsilon^3
\mathbf{w}^0_\pm(x)+\widetilde{\mathbf{w}}_\pm(x).\] 
The asymptotic behaviour of \(\mathbf{w}^0\) in the neighbourhood of \(x^0\) is
\(|x-x^0|^{-2}\) and \(\widetilde{\mathbf{w}}_\pm(x)=O(\epsilon^4)\). Thus by Lemma 4.3, 
or more precisely by the Hardy-type inequality,

\begin{eqnarray*}
&&\left|\Lambda^0_\pm(\pi)(\chi\widetilde{\mathbf{w}}_\pm,Z)_{\varpi_\epsilon}
+\epsilon^3\Lambda'_\pm(\pi,\psi)(\chi
\mathbf{w}^0_\pm,Z)_{\varpi_\epsilon}\right|\nonumber\\ 
&\leq&
c\epsilon^{7/2}\||x-x^0|^{-1}Z;{L_2(\varpi_\epsilon)}
\| \\ 
&\leq&
c\epsilon^{7/2}\|Z;{H^1(\varpi_\epsilon)}\|.\nonumber
\end{eqnarray*}

The evaluation of the remaining terms shall be made at the expense of an
additional factor \(\epsilon^3\), or due to the replacement of \(\mathbf{w}^0_\pm\) by
\(\widetilde{\mathbf{w}}_\pm\) leads to an improvement in the estimates.

Collecting obtained inequalities we obtain that \(\tau\leq c \epsilon^{7/2}\).
Thus by Lemma 4.4 the operator 
\({\cal}B^\epsilon(\pi+\psi\epsilon^3)\)
has just two eigenvalues \(\beta_\pm(\pi+\psi\epsilon^3)\) in the interval
\([\ell_\pm^{-1}- c\epsilon^{3/2},\ell_\pm^{-1}+ c\epsilon^{3/2}]\). 

Now the relationship (\ref{beta}) of the spectral parameters of the operator \({\mathcal{B}}\)
gives the result. 

\begin{theorem}
There exist \(C>0\), \(\epsilon_0>0\) and \(\psi_0>0\) such that for all
\(\epsilon\in(0,\epsilon_0]\)  
under the condition \(|\psi|\epsilon^{3/2}<\psi_0\) the problem
(\ref{form3})-(\ref{form5}) with \(\eta=\psi+\pi\epsilon^3\) has two eigenvalues
satisfying the inequality
\begin{equation*}
\left|\Lambda^\epsilon_\pm(\pi+\psi\epsilon^3)-M_1-\pi^2-\epsilon^3\Lambda'_{\pm
}(\pi,
\psi)\right|\leq C \epsilon^{7/2}\,,
\end{equation*} 
where \(\Lambda'_\pm\) is taken from (\ref{L'psi}).
\end{theorem}

\subsection{The detection of the gap.}

To detect a gap in the spectrum we apply the results of previous computations.
First, if \(|\eta-\pi|> \psi_0\epsilon^{3/2}\) with \(\psi_0>0\), then by means of
Theorem 4.2 we will show that the interval
\begin{equation}
\label{interval}
(M_1+\pi^2+(F_0-|F_1|)\epsilon^3 +c\epsilon^{7/2},
M_1+\pi^2+(F_0+|F_1|)\epsilon^3-c\epsilon^{7/2})
\end{equation}
is free of spectrum. Otherwise, in the case $|\eta-\pi|\leq
\psi_0\epsilon^{3/2}$
we make use of much more elaborated asymptotic result of Theorem 4.6 to
show that in the interval (\ref{interval}) there is no eigenvalues.
This provides us the existence of the  spectral gap.

To realise this scheme, we first observe that, if with some \(\psi_0>0\)
\[\eta\in (0, \pi-\psi_0\epsilon^{3/2}]\cup[\pi+\psi_0\epsilon^{3/2},2\pi)\,, \]
then by Theorem 4.2 
\begin{equation}\label{one}
\Lambda_1^\epsilon (\eta)\leq \Lambda^0_1(\eta)+C\epsilon^3
\leq M_1 +\pi^2-2\pi|\psi_0\epsilon^{3/2}|+\psi_0^2\epsilon^3+C\epsilon^3 
\end{equation}
and choosing \(\delta\in (0,\frac12]\) in Theorem 4.2
\begin{equation}\label{two}
\Lambda_2^\epsilon (\eta)\geq
\Lambda_2^0 (\eta)-C\epsilon^{3-\alpha}\geq  M_1
+\pi^2+2\pi|\psi_0\epsilon^{3/2}|+\psi_0^2\epsilon^3-C\epsilon^{3-\alpha}\,.
\end{equation}
Because of the terms \(\pm2\pi \psi_0\epsilon^{\frac32}\) on the right hand side of 
(\ref{one}) and (\ref{two}) we may choose \(\epsilon>0\) so small that
\[
\Lambda_1^\epsilon (\eta)<M_1+\pi^2+(F_0-|F_1|)\epsilon^3 +c\epsilon^{7/2}
\]
and
\[
\Lambda_2^\epsilon (\eta)>M_1+\pi^2+(F_0+|F_1|) \epsilon^3 -c\epsilon^{7/2} \,,
\]
and therefore the interval (\ref{interval}) does not include any eigenvalues.

If \(\eta=\pi+ \psi\epsilon^3 \in [\pi-\psi_0\epsilon^{3/2},
\pi+\psi_0\epsilon^{3/2}] \), then by Theorem 4.6
we have

\begin{eqnarray}
\Lambda_1^\epsilon (\eta)&\leq& M_1 +\pi^2+\epsilon^3 
\left(F_0 -\sqrt{|F_1|^2+4\pi^2\psi^2}\right) +c\epsilon^{7/2}\\ 
&\leq& 
M_1+\pi^2+(F_0-|F_1|)\epsilon^3 +c\epsilon^{7/2}\nonumber
\end{eqnarray}
and
\begin{eqnarray}
\Lambda_{2}^\epsilon (\eta)&\geq& M_1 +\pi^2+\epsilon^3 
\left(F_0 +\sqrt{|F_1|^2+4\pi^2\psi^2}\right) -c\epsilon^{7/2}\\
&\geq& 
M_1+\pi^2 +(F_0+|F_1|)\epsilon^3  - c\epsilon^{7/2}\nonumber
\end{eqnarray}
so interval (\ref{interval}) also is free of eigenvalues and we can formulate the
main result of the paper.

\begin{theorem} Assume that \(|F_1|>0\). Then
there exist positive numbers \(\epsilon_0\) and \(c_0\) such
that,
for \(\epsilon \in (0, \epsilon_0]\),
the essential spectrum of the problem \eqref{form1}-\eqref{form2} or
\eqref{dir-1} has a gap of length
\(\ell(\epsilon)\),
\[
\left|\ell(\epsilon)-2|F_1| \epsilon^3\right|<2 c_0 \epsilon^{7/2} 
\]
located just after the first segment \(\Upsilon_1(\epsilon)\). Moreover, the gap includes the interval
\[
(M_1+\pi^2 +(F_0-|F_1|)\epsilon^3-c_0\epsilon^{7/2}, 
M_1+\pi^2+(F_0+|F_1|)\epsilon^3+c_0\epsilon^{7/2})\,.
\] 
\end{theorem}

\subsection{Asymptotic formulae for the position and size of the spectral gap}

We proceed with discussing the Neumann problem (\ref{form1})-(\ref{form2}). 
Based on the relations (\ref{oulu1}) we first conclude that the condition \(|F_1|>0\) of Theorem 4.7
is always fulfilled, so that the gap
\(I(\epsilon)=(A_{-}(\epsilon),A_{+}(\epsilon))\)
opens in the spectrum. Moreover, its length is 
\[
 2\pi^2\epsilon^3\frac{\textrm{meas}_3(\theta)-Q_{33}}{\textrm{meas}_2(\omega)}+O(\epsilon^\frac72),
\]
where the endpoints admit the asymptotic form
\begin{eqnarray*}
A_{-}(\epsilon)&=&\pi^2+2\pi^2\epsilon^3\frac{Q_{33}}{\textrm{meas}_2(\omega)}+O(\epsilon^\frac72),\\
A_{+}(\epsilon)&=&\pi^2+2\pi^2\epsilon^3\frac{\textrm{meas}_3(\theta)}{\textrm{meas}_2(\omega)}+O(\epsilon^\frac72).
\end{eqnarray*}
Since \(Q_{33}<0\) (see Remark 3.1), the gap covers the point \(\pi^2\).

For the mixed boundary value problem (\ref{dir-1}) let \(y^0\) be the maximum point of 
the positive eigenfunction \(V_1\).
Then \(\nabla_y V_1(y^0)=0\) and the formula (\ref{oulu2}) leads to a 
similar simple situation as for the Neumann problem, because then the number
\begin{eqnarray*}
 e^{2\pi i z_0}F_1&=&
\nabla_y V_1(y^0)^\top Q'\nabla_y V_1(y^0)+
(M_1+\pi^2(\textrm{meas}_3(\theta)-Q_{33}))|V_1(y^0)|^2\\
&-&2\pi i\big(Q_{31}V_1(y^0)\frac{\partial V_1(y^0)}{\partial y_1}+
Q_{32}V_1(y^0)\frac{\partial V_1(y^0)}{\partial y_2}\big)\\
&=&(M_1+\pi^2(\textrm{meas}_3(\theta)-Q_{33})|V_1(y^0)|^2
\end{eqnarray*}
is for surely positive. The interval \(I(\epsilon)\) has the endpoints
\begin{eqnarray*}
A_{-}(\epsilon)&=&M_1+\pi^2+2\pi^2\epsilon^3Q_{33}|V_1(y^0)|^2+O(\epsilon^\frac72),\\
A_{+}(\epsilon)&=&M_1+\pi^2+2\epsilon^3(M_1+\pi^2\textrm{meas}_3(\theta))|V_1(y^0)|^2+O(\epsilon^\frac72)
\end{eqnarray*}
and its length is 
\[
2\epsilon^3(M_1+\pi^2(\textrm{meas}_3(\theta)-Q_{33}))|V_1(y^0)|^2+O(\epsilon^\frac72).
\]
The point \(M_1+\pi^2\) is contained in the interval \(I(\epsilon)\).

Theorem 4.7 ensures the opening of the in the case \(|F_1|>0\). However, in the mixed boundary value 
problem varying the position of the point \(y^0\)
in \(\omega\) we may come across the relation \(F_1=0\). In this case the asymptotic
formulae obtained here do not allow us to conclude the existence of the gap in the spectrum.

To show that there exist points for which \(|F_1|=0\), we assume that \(Q_{12}=Q_{13}=0\), 
for example due to the central symmetry of the set \(\theta\). Then the expression
\[ e^{2\pi i z_0}F_1=
\nabla_y V_1(y^0)^\top Q'\nabla_y V_1(y^0)+
(M_1+\pi^2(\textrm{meas}_3(\theta)-Q_{33}))|V_1(y^0)|^2\]
is real. Moreover,if \(y^0\) is the maximum point of \(V_1\), then the expression is positive; 
but it becomes negative when \(y^0\) locates close to the boundary \(\partial \omega\) 
where the homogeneous Dirichlet condition is imposed. 
This is due to the fact that the second term, which is positive, diminishes, while the first term remains 
strictly negative all the time according to Remark 3.1, because the normal derivative 
\(\partial_n V_1\)  is negative everywhere on \(\partial \omega\). Hence there must be a subset 
in \(\omega\) where \(F_1=0\). To prove or disprove the existence
of a spectral gap one needs to construct higher order terms in the asymptotic ans\"atze.

\subsection*{Acknowledgements} The first author was supported by the Chebyshev Laboratory
under RF government grant 11.G34.31.0026. The second author was supported
by Russian Foundation on Basic Research grant 09-01-00759.

\end{document}